\documentclass[a4paper, 12pt]{amsart}
\usepackage{amsmath}
\usepackage{graphicx}
\usepackage{amsfonts}
\usepackage{amssymb}

\newtheorem{theorem}{Theorem}[section]
\newtheorem{proposition}[theorem]{Proposition}
\newtheorem{lemma}[theorem]{Lemma}
\newtheorem{corollary}[theorem]{Corollary}
\theoremstyle{definition}
\newtheorem{definition}[theorem]{Definition}
\newtheorem{remark}[theorem]{Remark}

\newcommand{\id}{\operatorname{id}}
\newcommand{\Tr}{\operatorname{Tr}}

\makeatletter
\@namedef{subjclassname@2020}{%
	\textup{2020} Mathematics Subject Classification}
\makeatother

\begin{document}
\date{2021-12-22}
\title[Wasserstein distance between systems]{Wasserstein distance between noncommutative dynamical systems}
\author{Rocco Duvenhage}
\address{Department of Physics\\
University of Pretoria\\
Pretoria 0002\\
South Africa}
\email{rocco.duvenhage@up.ac.za}
\subjclass[2020]{Primary: 49Q22. Secondary: 46L55, 81S22.}

\begin{abstract}
We study a class of quadratic Wasserstein distances on spaces consisting of
generalized dynamical systems on a von Neumann algebra. We emphasize how
symmetry of such a Wasserstein distance arises, 
but also study the asymmetric case. 
This setup is illustrated in the context of reduced dynamics, 
and a number of simple examples are also presented.

\end{abstract}
\maketitle

\section{Introduction}


This paper studies a transport plan based approach to quadratic Wasserstein
distances on spaces consisting of dynamical systems on a fixed von Neumann
algebra, where each system is equipped with a state invariant under the
dynamics. Such distances are defined on the states, in line with the usual
measure theoretic definition, but with the sets of allowed transport plans
``regulated'' by the dynamics. This is done by certain balance conditions
imposed on the transport plans using the dynamics of two systems, which could
also be viewed as covariance conditions. This indeed leads to distances on the
space of dynamical systems, which can be metrics, pseudometrics or asymmetric
(pseudo)metrics, depending on the assumptions.

The approach taken here builds on \cite{D} and rests on bimodule ideas flowing
from the Tomita-Takesaki theory of von Neumann algebras. We can therefore
refer to it as a bimodule approach to Wasserstein metrics. It is analogous to
an approach taken in \cite{BV} for traces on C*-algebras in the context of
free probability.

One of our motivations for studying Wasserstein distances in this setting, is
to apply it to quantum detailed balance, specifically standard quantum
detailed balance conditions \cite{FU, FR}. In particular, to have a natural
measure of how far a system is from another satisfying detailed balance. The
notion of balance, mentioned above, was originally introduced in \cite{DS2}
with this goal in mind. Our eventual aim with Wasserstein distances is to make
that paper's approach more precise. This line of investigation will, however,
be pursued elsewhere.

The second reason to study Wasserstein distances between noncommutative
dynamical systems, rather than just states, is that the mentioned balance
conditions can significantly reduce the set of transport plans that need to be
considered in determining the distance. In this sense it is a simplification
that can ease the study of Wasserstein distances in concrete examples, which
in turn can give some insight into the nature and general behaviour of
Wasserstein distances in the bimodule approach. This paper indeed investigates
a handful of simple examples.

Another motivation is from ideas appearing in classical ergodic theory, in
particular in relation to the theory of joinings, the latter which now also
has a noncommutative counterpart \cite{D1, D2, D3, BCM, BCM2, BCM3}. In the
classical theory, these ideas originate with Ornstein's $\bar{d}$-metric
\cite{O, Vas}, developed further in \cite{GNS, RS}. It seems plausible that
the Wasserstein distances developed in this paper could have applications to
the theory of noncommutative joinings, though this will be left for future work.

In our formalism, we allow multiple dynamics in each system, with the dynamics
having minimal structure. For example, we do not assume the semigroup
property, the motivation being to allow for non-Markovian dynamics. We argue
in favour of this setup by illustrating the formalism in the context of
reduced dynamics, namely the dynamics a system inherits from a larger system
of which it is part. For reduced dynamics, semigroup properties may not hold,
and it will also be seen that further dynamics, other than the dynamics of
main interest, come up naturally, with balance conditions having a key role to play.

In order to obtain symmetry of the Wasserstein distances, the modular group
(or modular dynamics) of the state of a system has to be included in the
set of dynamics of the system, and corresponding modular balance conditions
between systems imposed, also involving the KMS-duals of all the dynamics.
This essentially generalizes the approach taken in \cite{D}. One of the goals
of this paper is to gain a better understanding of the role of modular
dynamics and these modular balance conditions by comparing the cases with and
without inclusion of the modular balance conditions. Consequently we
investigate both (pseudo)metrics and asymmetric (pseudo)metrics. The inclusion
of modular dynamics in addition to the dynamics of main interest, in order to
attain symmetry, is also another motivating reason not to restrict the
formalism to one dynamics per system.

Simple examples will be presented in Section \ref{AfdVbe} to gain further
insight into the general formalism. In particular, it will be seen by example
that without the above mentioned modular balance conditions, symmetry can
indeed fail. The examples will also cast some light on the behaviour of
Wasserstein distances in relation to the dynamics of systems, including
modular dynamics. Regarding the latter, certain jumps in the value of the
distance in the examples, suggest that it may be more natural to drop the
modular balance conditions when we are interested in distances between the
states themselves. In other words, it seems possible that asymmetric
Wasserstein (pseudo)metrics on states could be more natural in the bimodule
approach. This possibility will not be pursued further in this paper though.
On the other hand, for essentially the same reasons, the examples also
indicate that Wasserstein distances between systems should be of value in
discerning qualitative differences between dynamics, while typically being
insensitive to perturbations of dynamics. This could be of use in classifying
systems according to certain qualitative properties, as is relevant in ideas
related to detailed balance and ergodic theory.

Although the theory is developed in quite a general von Neumann algebraic
setup, the examples will mostly be for low dimensional matrix cases. The
reason for this is that in general it is difficult to determine the relevant
sets of transport plans, but in low dimensions one can make progress on this
problem by elementary means. We nevertheless study one example with an
infinite dimensional algebra as well, namely the quantum (or noncommutative)
torus. As mentioned earlier, the prospects to successfully calculate
Wasserstein distances improve significantly as the set of allowed transport
plans become smaller. Here the balance conditions come into play, making the
Wasserstein distances accessible in our examples.

We pause for a moment to give a very brief bird's eye view of noncommutative
optimal transport between states, in particular Wasserstein distances, which
is currently growing rapidly, to provide further context and motivation for
the present paper. The early work by Biane and Voiculescu \cite{BV} in the
framework of free probability was already mentioned, and lead to further
developments such as \cite{HPU}. Other transport plan (or coupling) based
approaches to Wasserstein distances, can be found in work by Ning, Georgiou
and Tannenbaum \cite{NGT}, a series of papers by Golse, Mouhot, Paul and
co-workers, starting with \cite{GMP}, as well as \cite{dePT} by de Palma and
Trevisan, though these papers in most cases did not manage to obtain metrics.
Agredo and Fagnola \cite{AF} indeed pointed out pitfalls in this respect. The
dynamical approach \cite{BB} to transport, from one distribution to another in
the classical case, formed the basis for noncommutative versions appearing in
series of papers by Carlen and Maas, as well as Chen, Georgiou and Tannenbaum,
starting with \cite{CM1} and \cite{CGT} respectively, with infinite
dimensional extensions by Hornshaw \cite{H1} and Wirth \cite{W}. There have
been other approaches as well, for example by Ikeda \cite{Ik} and De Palma,
Marvian, Trevisan and Lloyd \cite{DMTL}. Progress continues unabated, with
recent work such as \cite{CEFZ, FGP, GJNS, GR}, to name but a few. There are
many other papers treating further developments, but hopefully the mentioned
papers and remarks give the reader an idea of the importance, scope and
possibilities of noncommutative (or quantum) optimal transport. A standard
introduction to the classical theory, on the other hand, is \cite{V1}.

As for the structure of the paper: After fixing a few conventions and
collecting some preliminary concepts in the next section, we proceed in
Section \ref{AfdW2} to define our dynamical systems and to obtain Wasserstein
metrics and pseudometrics on spaces of such systems. Their asymmetric
counterparts, which arise when we drop the modular balance conditions outlined
above, are also discussed. This is followed by a study of reduced dynamics in
this setting in Section \ref{AfdRed}, to illustrate the formalism. The paper
concludes with examples in Section \ref{AfdVbe}, with the goal of making a
number of ideas in the paper concrete in simple cases.

\section{Conventions and preliminaries}

We consider a $\sigma$-finite von Neumann algebra $M$ which is fixed
throughout this section and the next. Denote the set of faithful normal states
on $M$ by $\mathfrak{F}(M)$. We denote the modular group associated with
$\mu\in\mathfrak{F}(M)$ by $\sigma^{\mu}$. By the theory of standard forms
\cite{Ar74, Con74, Ha75} (also see \cite[Theorem 2.5.31]{BR1}) we can assume
that $M$ is in standard form, meaning that $M$ is a\ von Neumann algebra on a
Hilbert space $G$, with every faithful normal state $\mu\in\mathfrak{F}(M)$
given by a cyclic and separating vector $\Lambda_{\mu}\in G$ for $M$, namely
\[
\mu(a)=\left\langle \Lambda_{\mu},a\Lambda_{\mu}\right\rangle
\]
for all $a\in M$, which allows us to define a state 
$\mu'\in\mathfrak{F}(M')$ on the commutant of $M$ by
\[
\mu'(a')=\left\langle \Lambda_{\mu},a'\Lambda_{\mu}\right\rangle
\]
for all $a'\in M'$. The unit of a von Neumann will usually be
written as $1$, but in cases where multiple von Neumann algebras are involved,
we occasionally indicate the algebra as a subscript, i.e., $1_{M}$, for
clarification. Define
\[
j_{\mu}:=J_{\mu}(\cdot)^{\ast}J_{\mu}
\]
on $\mathcal{B}(G)$ in terms of the modular conjugation $J_{\mu}$ for $M$
associated to $\Lambda_{\mu}\in G$. Note that $\mu'=\mu\circ j_{\mu}$.

We use analogous notation and conventions for another von Neumann algebra $N$,
also in standard form on another Hilbert space $H$, since in Section
\ref{AfdRed} we occasionally need the duals described below for maps between
two different algebras. In addition, the added generality will clarify certain
points in Section \ref{AfdRed}.

We are going to make use of duals and KMS-duals of unital completely positive
(u.c.p.) maps. Duals were introduced in \cite{AC} (see \cite[Section 2]{DS2}
for a summary). KMS-duals and KMS-symmetry were studied and applied in
\cite{Pet, OPet, GL93, GL, C, FR, DS2}.

\begin{definition}
\label{KMS-duaal}Given a u.c.p. map $E:M\rightarrow N$ such that 
$\nu\circ E=\mu$ for some $\mu\in\mathfrak{F}(M)$ and $\nu\in\mathfrak{F}(N)$, 
we define its \emph{KMS-dual} (w.r.t. $\mu$ and $\nu$) as
\[
E^{\sigma}:=j_{\mu}\circ E'\circ j_{\nu}:N\rightarrow M
\]
in terms of the \emph{dual} (w.r.t. $\mu$ and $\nu$)
\[
E':N'\rightarrow M'%
\]
of $E$ defined by
\begin{equation}
\left\langle \Lambda_{\mu},aE'(b')\Lambda_{\mu}\right\rangle
=\left\langle \Lambda_{\nu},E(a)b'\Lambda_{\nu}\right\rangle
\label{duaalDef}
\end{equation}
for all $a\in M$ and $b'=N'$. (Consult \cite{AC} for the
theory behind such duals and \cite[Section 2]{DS2} for a summary.)
\end{definition}

Note that according to \cite[Proposition 3.1]{AC}, $E'$ is a u.c.p.
map satisfying $\mu'\circ E'=\nu'$. Correspondingly
$E^{\sigma}$ is u.c.p. and $\mu\circ E^{\sigma}=\nu$, while 
$(E^{\sigma})^{\sigma}=E$ follows from $(E')'=E$. We also point out that
under the assumptions in the definition above, the maps $E$, $E'$ and
$E^{\sigma}$ are normal, i.e., $\sigma$-weakly continuous, though this fact
will not play a direct role in our work.

A special case of particular interest to us, is $M=N$ and $\mu\circ E=\mu$,
where $E$ will be viewed as dynamics leaving the state $\mu$ invariant. In
such cases the dual and KMS-dual will always be with respect to $\mu$, namely
$E^{\sigma}:=j_{\mu}\circ E'\circ j_{\mu}$, with $E'$ defined
in terms of $\Lambda_{\mu}$ on both sides of (\ref{duaalDef}).

As in \cite{D}, we use the following basic notion as a starting point for
optimal transport.

\begin{definition}
\label{oordPlan}A \emph{transport plan} from $\mu\in\mathfrak{F}(M)$ to
$\nu\in\mathfrak{F}(N)$, is a state $\omega$ on the algebraic tensor product
$M\odot N'$ such that
\[
\omega(a\otimes1)=\mu(a) \text{ \ \ and \ \ } \omega(1\otimes b')=\nu'(b')
\]
for all $a\in M$ and $b'\in N'$. Denote the set of all
transport plans from $\mu$ to $\nu$ by $T(\mu,\nu)$.
Transport plans are also known as \emph{couplings}. 
\end{definition}

This is a direct extension of the corresponding classical notion, discussed
and motivated in \cite{V1}. But the commutant is introduced to fit into the
bimodule structure to come, and plays a central role in our setup.

As a notational
convention, a transport plan from $\mu$ to $\nu$ will usually be denoted by
$\omega$ as in the definition above, but from $\nu$ to $\xi$ by $\psi$, and
from $\mu$ to $\xi$ by $\varphi$, for any $\mu,\nu,\xi\in\mathfrak{F}(M)$.

Transport plans from $\mu$ to $\nu$ are in a one-to-one correspondence with
u.c.p. maps $E:M\rightarrow N$ such that $\nu\circ E=\mu$: Let
\[
\varpi_{N}:N\odot N'\rightarrow\mathcal{B}(H)
\]
be the unital $\ast$-homomorphism defined by extending 
$\varpi_{N}(b\otimes b')=bb'$. Note that
\[
\delta_{\nu}
:=\left\langle \Lambda_\nu,\varpi_{N}(\cdot)\Lambda_\nu\right\rangle
\]
is a transport plan from $\nu$ to itself. Then there is a unique map
\[
E_{\omega}:M\rightarrow N
\]
such that
\begin{equation}
\omega(a\otimes b')=\delta_{\nu}(E_{\omega}(a)\otimes b')
\label{E}
\end{equation}
for all $a\in M$ and $b'\in N'$. This map $E_{\omega}$ is
linear, normal, u.c.p., and satisfies
\[
\nu\circ E_{\omega}=\mu.
\]
Conversely, given a u.c.p. map $E:M\rightarrow N$ such that $\nu\circ E=\mu$,
it defines a transport plan $\omega_{E}$ from $\mu$ to $\nu$ by
\[
\omega_{E}(a\otimes b')=\delta_{\nu}(E(a)\otimes b'),
\]
which satisfies $E=E_{\omega_{E}}$. Technical details can be found in
\cite[Section 3]{DS2}. This correspondence appears, for example, in finite
dimensions in quantum information theory, where it is known as the
Choi-Jamio{\l}kowski duality (with Choi's version \cite{Choi} corresponding to
our setup), and also in the theory of noncommutative joinings \cite{BCM}.

\section{Wasserstein distances\label{AfdW2}}

Given these conventions and preliminaries, we proceed with only the one von
Neumann algebra $M$ in this section, to define two types of Wasserstein
distances $W$ and $W_{\sigma}$. The latter is a metric under suitable the
conditions, the former only an asymmetric metric.

We introduce the following series of definitions, forming the foundation of
our development.

\begin{definition}
\label{stelsel}A \emph{generalized system} on $M$ is given by 
$\mathbf{A}=\left(\alpha,\mu\right)$, where $\mu\in\mathfrak{F}(M)$, while $\alpha$
consists of the following: Let $\Upsilon$ be any set. To each 
$\upsilon\in\Upsilon$ corresponds a set $Z_{\upsilon}$ and 
\emph{generalized dynamics}
$\alpha_{\upsilon}$ on $M$, which is given by a u.c.p. map
\[
\alpha_{\upsilon,z}:M\rightarrow M
\]
for every $z\in Z_{\upsilon}$, such that
\[
\mu\circ\alpha_{\upsilon,z}=\mu
\]
for all $z\in Z_{\upsilon}$ and $\upsilon\in\Upsilon$. We then write
\[
\alpha=\left(  \alpha_{\upsilon}\right)  _{\upsilon\in\Upsilon}.
\]
\end{definition}

Such a generalized system wil be refered to simply as a ``system'' in the
sequel. We can view each $Z_{\upsilon}$ as a set of ``points in time'' in an
abstract sense. Each $\alpha_{\upsilon}$ is viewed as dynamics, so in effect
we have a set of dynamical systems on $M$, indexed by $\upsilon$. We do not
assume any structure, for example semigroup structure, on $Z_{\upsilon}$. In
part this is because we do not need any structure, but also since in
applications, we want to allow for dynamics that may not have semigroup
structure, as will be seen in Section \ref{AfdRed}.

Furthermore, we allow multiple dynamics, since in order to obtain symmetry of a
Wasserstein distance, we need to include the \emph{modular dynamics} (i.e.,
the modular group) in any case, while there are other natural dynamics that
can also play a role, an example of which will be seen in Section \ref{AfdRed}.

We can always include the modular dynamics among the the $\alpha^{\upsilon}$'s,
however, to emphasize the role it plays, it will be handled separately in
this section.

In the next section, multiple von Neumann algebras will be involved, in which
case we add the algebra to the notation for the system, i.e., 
$\mathbf{A}=\left(M,\alpha,\mu\right)$.

For the remainder of this section, we fix $\Upsilon$ and the $Z_{\upsilon}$'s.
The following notational convention will be used: $\mathbf{A}$ will denote
$\left(\alpha,\mu\right)$, as in the definition above, and similarly we
write
\[
\mathbf{B}=
\left(\beta,\nu\right)  \text{ \ \ and \ \ } \mathbf{C}=\left(\gamma,\xi\right)
\]
for systems on $M$, for the same sets $\Upsilon$ and the $Z_{\upsilon}$ used
in $\mathbf{A}$. Let
\[
X
\]
denote the set of all such systems $\mathbf{A}$ on $M$. Our Wasserstein
distances will be defined on $X$.

\begin{definition}
The \emph{KMS-dual} of a system $\mathbf{A}$, is the system on $M$ given by
\[
\mathbf{A}^{\sigma} = \left(\alpha^{\sigma},\mu\right)
\]
where 
$\alpha^{\sigma}=\left(\alpha_{\upsilon}^{\sigma}\right)_{\upsilon\in\Upsilon}$, 
and $\alpha_{\upsilon}^{\sigma}$ is given by
$\alpha_{\upsilon,z}^{\sigma}$, i.e., we take the KMS-dual of each
$\alpha_{\upsilon,z}$ w.r.t. $\mu$.
\end{definition}

To obtain our Wasserstein distances on $X$, we are going to use restricted
sets of transport plans. To define them, we use a property
which was called balance in \cite{DS2}:

\begin{definition}
\label{balans}We say that $\mathbf{A}$ and $\mathbf{B}$ (in this order) are in
\emph{balance} with respect to $\omega\in T(\mu,\nu)$, written as
\[
\mathbf{A}\omega\mathbf{B},
\]
if%
\[
\left(\alpha_{\upsilon},\mu\right)\omega\left(\beta_{\upsilon},\nu\right)
\]
for all $\upsilon\in\Upsilon$, by which we mean that
\begin{equation}
\omega(\alpha_{\upsilon,z}(a)\otimes b')=
\omega(a\otimes\beta_{\upsilon,z}'(b')) 
\label{balItvKop}
\end{equation}
for all $a\in M$, $b'\in M'$ and $z\in Z_{\upsilon}$, in terms
of the dual defined in Definition \ref{KMS-duaal}.
\end{definition}

\begin{remark}
\label{balUitbr}Balance has an obvious extension (see \cite{DS2}) to the case
where $\mathbf{A}$ and $\mathbf{B}$ do not necessarily have the same von
Neumann algebra, say $M$ and $N$ respectively. One would then write 
$\left(M,\alpha_{\upsilon},\mu\right) \omega \left( N,\beta_{\upsilon},\nu\right)$
to mean (\ref{balItvKop}) for all $a\in M$, $b'\in N'$. In the
present paper this extension is not technically needed, though in Section
\ref{AfdRed} the extended version will briefly be used to clarify certain points.
\end{remark}

An important example of balance is for the modular dynamics of the states,
namely
\[
\left(\sigma^{\mu},\mu\right) \omega \left(\sigma^{\nu},\nu\right),
\]
i.e.,
\[
\omega(\sigma_{t}^{\mu}(a)\otimes b')=\omega(a\otimes\sigma_{t}^{\nu'}(b'))
\]
for all $a\in M$, $b'=M'$ and $t\in\mathbb{R}$. Note that here
we used the fact that 
$\left(\sigma_{t}^{\nu}\right)' = \sigma_{t}^{\nu'}$, 
i.e., the modular group associated with $\nu'$.
%

Our restricted sets of transport plans are then defined as follows:

\begin{definition}
\label{T(A,B)}The set of \emph{transport plans} from $\mathbf{A}$ to
$\mathbf{B}$ is
\[
T(\mathbf{A},\mathbf{B}):=
\left\{\omega\in T(\mu,\nu):\mathbf{A}\omega\mathbf{B}\right\}.
\]
The set of \emph{modular} transport plans from $\mathbf{A}$ to $\mathbf{B}$
is
\[
T_\sigma(\mathbf{A},\mathbf{B}):=
\left\{
\omega\in T(\mathbf{A},\mathbf{B}):
\mathbf{A}^\sigma\omega\mathbf{B}^\sigma
\text{ and }
(\sigma^\mu,\mu)\omega(\sigma^\nu,\nu)
\right\}.
\]
\end{definition}

Note that we always have 
$\mu\odot\nu'\in T_\sigma(\mathbf{A},\mathbf{B})$. 
The conditions $\mathbf{A}^{\sigma}\omega\mathbf{B}^{\sigma}$
and $\left(\sigma^\mu,\mu\right)  \omega \left(\sigma^\nu,\nu\right)$
will collectively be called the \emph{modular balance conditions}.

\begin{remark}
\label{TvsTsig}
Since $(\sigma_{t}^{\mu})^{\sigma}=\sigma_{-t}^{\mu}$, the condition 
$\left(\sigma^\mu,\mu\right)  \omega \left(\sigma^\nu,\nu\right)$ can also be
written as 
$(\sigma^\mu,\mu)^\sigma\omega(\sigma^\nu,\nu)^\sigma$.
Hence, if the modular dynamics are
included in $\mathbf{A}$ and $\mathbf{B}$, at the same index value $\upsilon$,
then the condition 
$\left(\sigma^{\mu},\mu\right)  \omega\left(\sigma^{\nu},\nu\right)$ 
becomes redundant in $T_{\sigma}(\mathbf{A},\mathbf{B})$.
Similarly, for any $\ast$-automorphism $\tau$ of $M$ such that $\mu\circ\tau=\mu$, 
we have 
$\tau^{\sigma}=\tau^{-1}$. If for each pair $\left(  \upsilon,z\right)$, it holds that
$\alpha_{\upsilon,z}$ and $\beta_{\upsilon,z}$ are either both $\ast$-automorphisms, or both \emph{KMS-symmetric}, i.e., 
$\alpha_{\upsilon,z}^{\sigma}=\alpha_{\upsilon,z}$ and $\beta_{\upsilon,z}^{\sigma}=\beta_{\upsilon,z}$,
then
\[
T_{\sigma}(\mathbf{A},\mathbf{B})=T(\mathbf{A},\mathbf{B})
\]
if the modular dynamics are included in $\mathbf{A}$ and $\mathbf{B}$ as above.
\end{remark}

By \cite[Theorem 4.1]{DS2} we can express $\mathbf{A}\omega\mathbf{B}$ as
\begin{equation}
E_{\omega}\circ\alpha_{\upsilon,z}=\beta_{\upsilon,z}\circ E_{\omega}
\label{BalE}
\end{equation}
for all $z\in Z_{\upsilon}$ and $\upsilon\in\Upsilon$, in terms of $E_{\omega}$ 
given by (\ref{E}). This formulation of balance as a covariance is
often useful, as it is in joinings \cite{BCM}.

\begin{definition}
\label{I}Given $k_{1},...,k_{n}\in M$, and writing 
$k=\left(  k_{1},...,k_{n}\right)  $, the associated \emph{transport cost function} 
$I_{k}$,
which gives the cost of transport $I_{k}(\omega)$ from $\mu\in\mathfrak{F}(M)$
to $\nu\in\mathfrak{F}(M)$ for the transport plan $\omega\in T(\mu,\nu)$, is
defined to be
\begin{equation}
I_{k}(\omega)=
\sum_{l=1}^{n}
\left[  
\mu(k_{l}^{\ast}k_{l})+\nu(k_{l}^{\ast}k_{l})-
\nu(E_{\omega}(k_{l})^{\ast}k_{l})-\nu(k_{l}^{\ast}E_{\omega}(k_{l}))
\right]. 
\label{Ialt}
\end{equation}
\end{definition}

This formulation uses the ideas we have set up so far. An equivalent, but more
suggestive formulation, in terms of the cyclic representations 
$(H_{\mu}^{\omega},\pi_{\mu}^{\omega},\Omega)$ 
and 
$(H_{\nu}^{\omega},\pi_{\nu}^{\omega},\Omega)$ 
of $\left(  M,\mu\right)  $ and $\left(  M,\nu\right)  $,
respectively, inherited from the cyclic representation 
$(H_{\omega},\pi_{\omega},\Omega)$ of $(M\odot M',\omega)$, 
is
\[
I_{k}(\omega)=
\left\|  \pi_\mu^\omega(k)\Omega-\pi_\nu^\omega(k)\Omega\right\|_{\oplus\omega}^2,
\]
where we have written
\[
\pi_{\mu}^{\omega}(k)\Omega
\equiv
\left(  \pi_\mu^\omega(k_{1})\Omega,...,\pi_\mu^\omega(k_{n})\Omega \right)  
\in
\bigoplus_{l=1}^{n}H_\omega,
\]
while $\left\|  \cdot\right\|  _{\oplus\omega}$ denotes the norm on
$\bigoplus_{l=1}^{n}H_{\omega}$. I.e.,
\[
I_{k}(\omega)=\sum_{l=1}^{n}\left\|  \pi_{\mu}^{\omega}(k_{l})\Omega-\pi_{\nu
}^{\omega}(k_{l})\Omega\right\|  _{\omega}^{2}.
\]
More detail regarding these representations can be found in \cite{D}. From
this formulation it is clear that $I_{k}(\omega)\geq0$, but it can also be
seen directly from (\ref{Ialt}):
For all $a,b\in M$,
\begin{align*}
&  
\mu(a^{\ast}a)+\nu(b^{\ast}b)-\nu(E_{\omega}(a)^{\ast}b)-\nu(b^{\ast}E_{\omega}(a))\\
&  
=
\nu\left(\left|  b-E_{\omega}(a)\right|  ^{2}\right) +
\mu\left(\left|  a\right| ^{2}\right) -
\nu\left(\left| E_{\omega}(a)\right|^{2}\right) \\
&  
\geq
\nu\left(  \left|  b-E_{\omega}(a)\right|  ^{2}\right) \\
&  
\geq
0
\end{align*}
by Kadison's inequality.

In what follows, by a \emph{distance function} on $X$, we simply mean a
function $d:X\times X\rightarrow\mathbb{R}$. Such a function may be a metric.
However, we are also interested in whether a distance function $d$ is a
\emph{pseudometric}, which means that it satisfies the triangle inequality, is
symmetric, and obeys $d(x,x)=0$ and $d(x,y)\geq0$, but could allow $d(x,y)=0$
for $x\neq y$. Similarly for the asymmetric cases, to which we return at the
end of this section.

The specific distance functions we study in this paper, will collectively be
called \emph{Wasserstein distances}. They are the distance functions $W$ and
$W_{\sigma}$ on the space $X$ of systems, given by our main definition below.

\begin{definition}
\label{W2}Given $k_{1},...,k_{n}\in A$, we define the associated
\emph{Wasserstein distance} $W$\ on $X$ by
\[
W(\mathbf{A},\mathbf{B}):=
\inf_{\omega\in T(\mathbf{A},\mathbf{B})}I_{k}(\omega)^{1/2},
\]
and the associated \emph{modular Wasserstein distance} $W_{\sigma}$\ on $X$
by
\[
W_{\sigma}(\mathbf{A},\mathbf{B}):=
\inf_{\omega\in T_{\sigma}(\mathbf{A},\mathbf{B})}I_{k}(\omega)^{1/2},
\]
for all $\mathbf{A},\mathbf{B}\in X$, in terms of Definition \ref{I}.
\end{definition}

These are also called a Wasserstein distances of order $2$, or a quadratic
Wasserstein distances. Since we focus exclusively on the quadratic case in
this paper, we do not include a subscript $2$ as is standard notation in the
classical case. More complete notation would be to include the $k$, say as
$W^{(k)}$ and $W_{\sigma}^{(k)}$, but no confusion should arise.

Clearly
\begin{equation}
W(\mathbf{A},\mathbf{B})\leq W_{\sigma}(\mathbf{A},\mathbf{B})
\label{Wfyner}
\end{equation}
for all $\mathbf{A},\mathbf{B}\in X$, while in specific examples it should be
easier to determine $W_{\sigma}(\mathbf{A},\mathbf{B})$, as the additional
balance conditions simplify finding the relevant (and smaller) set of
transport plans $T_{\sigma}(\mathbf{A},\mathbf{B})$.

The central results of this section are the following two theorems, regarding
the metric properties of $W_{\sigma}$ and $W$ respectively, with $W_{\sigma}$
enjoying symmetry, but $W$ not. The triangle inequality emerges from the
natural $M$-$M$-bimodule structure of the GNS Hilbert spaces $H_{\omega}$ of
the transport plans $\omega$, and their relative tensor products 
$H_{\omega}\otimes_{\nu}H_{\psi}$. 
These products are treated in \cite[Section IX.3]{T2}, 
but see also \cite{Fal} and the early work \cite{Sa}. They originated
with Connes' largely unpublished work on correspondences, though
\cite[Appendix V.B]{Con94} gives a partial exposition.

\begin{theorem}
\label{metriek}
Let $W_{\sigma}$ be the modular Wasserstein distance on $X$
associated to $k_{1},...,k_{n}\in M$.

(a) Then $W_{\sigma}$ is a pseudometric and its value 
$W_{\sigma}(\mathbf{A},\mathbf{B})$ 
is reached by some transport plan $\omega\in T_{\sigma}(\mathbf{A},\mathbf{B})$, 
i.e., optimal modular transport plans
always exist.

(b) If in addition we assume that 
$\{k_{1}^{\ast},...,k_{n}^{\ast}\}=\{k_{1},...,k_{n}\}$ 
and that $M$ is generated by $k_{1},...,k_{n}$, then
$W_{\sigma}$ is a metric.
\end{theorem}

\begin{proof}
(a) By its definition, $W_{\sigma}$ is real-valued and never negative. Also
note that $W_{\sigma}(\mathbf{A},\mathbf{A})=0$ from (\ref{Ialt}) with
$\omega=\delta_{\mu}$, which is an element of 
$T_{\sigma}(\mathbf{A},\mathbf{A})$, since $E_{\omega}=\id_M$, 
trivially delivering all the balance requirements.

\emph{The triangle inequality.} For 
$\omega\in T_{\sigma}(\mathbf{A},\mathbf{B})$ 
and 
$\psi\in T_{\sigma}(\mathbf{B},\mathbf{C})$, 
we set
$\varphi=\omega\circ\psi$, defined via
\[
E_{\omega\circ\psi}=E_{\psi}\circ E_{\omega}.
\]
Note that $\varphi\in T_{\sigma}(\mathbf{A},\mathbf{C})$, since
\[
E_{\varphi}\circ\alpha_{\upsilon,z}=
E_{\psi}\circ E_{\omega}\circ\alpha_{\upsilon,z}=
E_{\psi}\circ\beta_{\upsilon,z}\circ E_{\omega}=
\gamma_{\upsilon,z}\circ E_{\psi}\circ E_{\omega}=
\gamma_{\upsilon,z}\circ E_{\varphi},
\]
hence $\mathbf{A}\varphi\mathbf{C}$, while similarly 
$\mathbf{A}^{\sigma}\varphi\mathbf{C}^{\sigma}$ 
and 
$(\sigma^{\mu},\mu)\varphi\left(  \sigma^{\xi},\xi\right)  $. 
As in the proof of \cite[Proposition 4.3]{D} we have
\[
I_{k}(\varphi)^{1/2}\leq I_{k}(\omega)^{1/2}+I_{k}(\psi)^{1/2},
\]
by employing the triangle inequality in 
$\bigoplus_{l=1}^{n}(H_{\omega}\otimes_{\nu}H_{\psi})$, 
and using properties of the relative tensor product
of the $M$-$M$-bimodules $H_{\omega}$ and $H_{\psi}$ (the required continuity
properties making them $M$-$M$-bimodules, were shown to hold in \cite[Theorem
3.3]{BCM}). Now take the infimum on the left over all of 
$T_{\sigma}(\mathbf{A},\mathbf{C})$, which includes the compositions 
$\omega\circ\psi$
for all 
$\omega\in T_{\sigma}(\mathbf{A},\mathbf{B})$ and 
$\psi\in T_{\sigma}(\mathbf{B},\mathbf{C})$, 
followed in turn by the infima over all 
$\omega\in T_{\sigma}(\mathbf{A},\mathbf{B})$ and 
$\psi\in T_{\sigma}(\mathbf{B},\mathbf{C})$ on the right.

\emph{Symmetry.} Recall from \cite[Lemma 5.2]{D} that a u.c.p. map
$E:M\rightarrow M$ satisfying $\nu\circ E=\mu$, has a Hilbert space
representation as a contraction $U:G\rightarrow G$ defined through
$Ua\Lambda_{\mu}=E(a)\Lambda_{\nu}$ for all $a\in M$, such that the following
equivalence holds:
\[
\mu(aE^{\sigma}(b))=\nu(E(a)b)
\]
for all $a,b\in M$, if and only if
\[
J_{\nu}U=UJ_{\mu}.
\]

We now apply this to $E_{\omega}$ and its Hilbert space representation
$U_{\omega}$, for $\omega\in T_{\sigma}(\mathbf{A},\mathbf{B})$. Since
$\omega$ satisfies 
$\left(  \sigma^{\mu},\mu\right)  \omega\left(  \sigma^{\nu},\nu\right)  $, 
it follows from \cite[Theorem 4.1]{DS2} that
$\Delta_{\nu}^{it}U_{\omega}=U_{\omega}\Delta_{\mu}^{it}$, where 
$\Delta_{\mu}$ and $\Delta_{\nu}$ are the modular operators associated to 
$\Lambda_{\mu}$ and $\Lambda_{\nu}$ respectively.
Consequently $J_{\nu}U_{\omega}=U_{\omega}J_{\mu}$. It therefore follows that
\begin{align*}
&  
\mu(a^{\ast}a)+\nu(b^{\ast}b)-\nu(E_{\omega}(a)^{\ast}b)-
\nu(b^{\ast}E_{\omega}(a))\\
&  
=\nu(b^{\ast}b)+\mu(a^{\ast}a)-\mu(E_{\omega}^{\sigma}(b)^{\ast}a)-
\mu(a^{\ast}E_{\omega}^{\sigma}(b))\\
&  
=\nu(b^{\ast}b)+\mu(a^{\ast}a)-\mu(E_{\omega^{\sigma}}(b)^{\ast}a)-
\mu(a^{\ast}E_{\omega^{\sigma}}(b))
\end{align*}
where $\omega^{\sigma}\in T(\nu,\mu)$ is determined by 
$E_{\omega^{\sigma}}=E_{\omega}^{\sigma}$ 
according to \cite[Section 4]{DS2}. Note that
$\omega^{\sigma}\in T_{\sigma}(\mathbf{B},\mathbf{A})$, since 
$\mathbf{A}^{\sigma}\omega\mathbf{B}^{\sigma}$ 
is equivalent to $\mathbf{B}\omega^{\sigma}\mathbf{A}$, 
and $\mathbf{A}\omega\mathbf{B}$ to 
$\mathbf{B}^{\sigma}\omega^{\sigma}\mathbf{A}^{\sigma}$, 
according to \cite[Corollary 4.6]{DS2},
while 
$\left(\sigma^\nu,\nu\right) \omega^\sigma\left( \sigma^\mu,\mu\right) $, 
since 
$E_{\omega^{\sigma}}\circ\sigma_{t}^{\nu}=
(\sigma_{-t}^{\nu}\circ E_{\omega})^{\sigma}=
(E_{\omega}\circ\sigma_{-t}^{\mu})^{\sigma}=
\sigma_{t}^{\mu}\circ E_{\omega^{\sigma}}$, 
due to 
$(\sigma_{t}^{\mu})^{\sigma}=\sigma_{-t}^{\mu}$.

The required symmetry 
$W_{\sigma}(\mathbf{A},\mathbf{B})=W_{\sigma}(\mathbf{B},\mathbf{A})$ 
now follows from (\ref{Ialt}) and Definition \ref{W2}
of $W_{\sigma}$, since for each $\omega$, every term in $I_{k}(\omega)$ is
equal to the corresponding term in $I_{k}(\omega^{\sigma})$, while
$(\omega^{\sigma})^{\sigma}=\omega$ because of $(E_{\omega}^{\sigma})^{\sigma
}=E_{\omega}$, giving a one-to-one correspondence between $T_{\sigma
}(\mathbf{A},\mathbf{B})$ and $T_{\sigma}(\mathbf{B},\mathbf{A})$, which means
we retain equality in the infima over $T_{\sigma}(\mathbf{A},\mathbf{B})$ and
$T_{\sigma}(\mathbf{B},\mathbf{A})$ respectively.

\emph{Optimal transport plans exist.}
We know (see for example \cite[Proposition 4.1]{D2}) that without loss we can
view each element of $T_{\sigma}(\mathbf{A},\mathbf{B})$ as a state on the
maximal C*-tensor product $A\otimes_{\text{max}}B'$, from which it
follows that $T_{\sigma}(\mathbf{A},\mathbf{B})$ is weakly* compact.
By the definition of $W_{\sigma}$ there is a sequence 
$\omega_{q}\in T_{\sigma}(\mathbf{A},\mathbf{B})$ such that 
$I_{k}(\omega_{q})^{1/2}\rightarrow W_{\sigma}(\mathbf{A},\mathbf{B})$, 
which necessarily has a weak*
cluster point $\omega\in T_{\sigma}(\mathbf{A},\mathbf{B})$. The existence of
an optimal transport plan is now obtained by the same approximation argument
as for \cite[Lemma 6.2]{D}, namely 
$I_{k}(\omega)^{1/2}=W_{\sigma}(\mathbf{A},\mathbf{B})$.

(b)
If $W_{\sigma}(\mu,\nu)=0$, then $\mu=\nu$ follows from the existence of an
optimal transport plan combined with \cite[Corollary 6.4]{D}.
\end{proof}

Dropping the modular balance conditions, the same proof, with minor
modifications (mostly simplifications) also delivers the corresponding result
for $W$ below, but without symmetry. We say that 
$d:X\times X\rightarrow\lbrack 0,\infty)$ is an \emph{asymmetric pseudometric}, if it satisfies the
triangle inequality and $d(x,x)=0$. If in addition $d(x,y)=0$ implies that
$x=y$, then we call $d$ an \emph{asymmetric metric}. The point being that we
do not assume symmetry, $d(x,y)=d(y,x)$, in these definitions.

\begin{theorem}
\label{asim}
Let $W$ be the Wasserstein distance on $X$ associated to
$k_{1},...,k_{n}\in M$.

(a) Then $W$ is an asymmetric pseudometric and its value 
$W(\mathbf{A},\mathbf{B})$ is reached by some transport plan 
$\omega\in T(\mathbf{A},\mathbf{B})$, i.e., optimal transport plans always exist.

(b) If in addition we assume that 
$\{k_{1}^{\ast},...,k_{n}^{\ast}\}=\{k_{1},...,k_{n}\}$ and that 
$M$ is generated by $k_{1},...,k_{n}$, then $W$ is an asymmetric metric.
\end{theorem}

In Section \ref{AfdVbe} it will indeed be
seen by simple example that in general $W$ does not possess symmetry.

We study both the \emph{modular Wasserstein pseudometric} $W_{\sigma}$ and the
\emph{asymmetric Wasserstein pseudometric} $W$ in the sequel. It tends to be
easier to prove results for $W$, as we do not need to take care of the modular
balance conditions. Consequently, as above, we provide more detailed arguments
regarding $W_{\sigma}$ in the next section.

\section{Reduced systems\label{AfdRed}}

The goal of this section is to illustrate that the formulation of Wasserstein
distances between systems in the previous section, in terms of balance
conditions, is natural. We do this in the context of reduced dynamics, where,
loosely speaking, we study the dynamics of a system $\mathbf{S}$, which is
interacting with another, called the \emph{reservoir}. The interacting
reservoir-system combination, will be called a \emph{composite system}
$\mathbf{A}$, with its dynamics due to the interaction called its
\emph{evolution}. In such cases $\mathbf{S}$ is referred to as an \emph{open
system}, but in this section we also call it a \emph{reduced system}, as its
dynamics is ``reduced'' from that of the composite system. We aim to determine
how Wasserstein distances on the set of reduced systems relate to Wasserstein
distances on the set of composite systems.

The literature on open sytems is vast. Standard textbook treatments can be
found in \cite{BP, Da}, for example. However, we presuppose as little as
possible background from this field, in order to make this section accessible
using only the framework set up so far.

We assume the reservoir and $\mathbf{S}$ to have fixed von Neumann algebras
$R$ and $S$ respectively, but allow different states and dynamics to obtain a
set of composite systems and a set of reduced systems (we do not need to
define the reservoirs as separate entities). We then make use of a transport
cost function (Definition \ref{I}) for which the $k_{l}$'s are in $1_R\otimes
S$, giving Wasserstein distances on the set $Y$ of reduced systems. In this
context we can also obtain natural Wasserstein distances on the set of
composite systems via a reduction. This reduction is performed by including an
appropriate conditional expectation of $R\otimes S$ onto $1_{R}\otimes S$ as
dynamics (a \emph{channel}) in the composite system.

In general
the reduced dynamics does not satisfy a semigroup property, even if that of
the composite system does. In this section we focus on the case where the
evolution of $\mathbf{A}$ does have the semigroup property. The reason for
this is solely to emphasize that even then the semigroup property in general
does not hold for $\mathbf{S}$.

If we include the modular dynamics, then both our composite system and reduced
system will have $\Upsilon=\{1,2\}$, with a semigroup $Z_{1}=\Gamma$ 
for the evolution, and $Z_{2}=\mathbb{R}$ for the modular dynamics.
However, we also consider an augmented form of a composite system, with the
dynamics supplemented by the conditional expectation mentioned above. In this
case one has $\Upsilon=\{1,2,3\}$, with the same $Z_{1}$ and $Z_{2}$ as for
the composite systems, and a one-point set $Z_{3}$ for the conditional expectation.

In this way, the present section demonstrates why the general setup of the
previous section does not assume semigroup (or any other) structure on the
sets $Z_{\upsilon}$, and why it allows multiple dynamics (indexed by
$\upsilon\in\Upsilon$). This section also emphasizes the balance conditions,
in this case in particular for the conditional expectation, to illustrate that
they play a natural role.

\begin{remark}
In this section we formulate everything with a view towards $W_{\sigma}$, that
is, with the modular dynamics and the required balance conditions ensuring
symmetry included. The same outline holds for $W$, but with simpler and
shorter arguments, hence we do not present them separately. Short remarks at
the end of this section will suffice.
\end{remark}

To make the discussion above more precise, the setup in this section is
as follows:
%

We set
\[
M=R\bar{\otimes}S,
\]
where $R$ and $S$ are both $\sigma$-finite von Neumann algebras in standard
form, fixed throughout this section, and $\bar{\otimes}$ denotes the von
Neumann algebraic tensor product. In this section we only consider product
states
\[
\mu=\mu_{R}\bar{\otimes}\mu_{S}
\]
on $M$, where $\mu_{R}\in\mathfrak{F}(R)$ and $\mu_{S}\in\mathfrak{F}(S)$.

In the previous section, the notation $\left( \alpha,\mu\right) $ was used
for a system, however, since there are several algebras involved in this
section, we include the algebra as well in the notation for the system, namely
$\left(  M,\alpha,\mu\right) $. We are interested in \emph{composite systems}
of the form
\begin{equation}
\mathbf{A} = \left(  M,\alpha,\sigma^{\mu},\mu\right)  =
\left( M,\alpha,\sigma^{\mu_R\bar\otimes\mu_S},\mu_{R}\bar\otimes\mu_S \right),
\label{basVorm}
\end{equation}
but will also require their \emph{augmented} form, given by
%
\[
\mathbf{A}^{\mathfrak{p}}=
\left(  M,\alpha,\sigma^{\mu},P_{1\otimes S}^{\mu_{R}},\mu\right)  ,
\]
where:

$\alpha$ is a semigroup of u.c.p. maps $\alpha_{g}:M\rightarrow M$ such that
$\mu\circ\alpha_{g}=\mu$ for all $g\in\Gamma$. Here the term semigroup means
that $\alpha_{gh}=\alpha_{g}\circ\alpha_{h}$ for all $g,h\in\Gamma$. We call
$\alpha$ the \emph{evolution} of $\mathbf{A}$.

$P_{1\otimes S}^{\mu_{R}}:M\rightarrow M$ is the conditional expectation onto
$1_{R}\otimes S$ defined by
\[
P_{1\otimes S}^{\mu_{R}}=
\left(  \mu_{R}1_{R}\right)  \bar{\otimes}\operatorname{id}_{S},
\]
where $1_{R}$ is the unit of $R$. In particular, 
$P_{1\otimes S}^{\mu_{R}}(r\otimes s)=\mu_{R}(r)1_{R}\otimes s$ 
for any elementary tensor $r\otimes s$
in $M$.

Let
\[
X_{\otimes}
\]
be the space of all composite systems on $M$ of the form in (\ref{basVorm}),
as described above. Similarly,
\[
X_{\otimes}^{\mathfrak{p}}
\]
will denote the space of augmented systems as defined above.
%

Note that we explicitly include the modular dynamics in the systems to ensure
symmetry of the Wasserstein distances, as seen in Section \ref{AfdW2}.

\begin{remark}
A standard physical situation is where $\Gamma=\mathbb{R}$, with $\alpha$ in
addition being a one-parameter group of $\ast$-automorphisms, i.e.,
$\alpha_{t}$ is a $\ast$-automorphism for every $t\in\mathbb{R}$ and
$\alpha_{0}=\id_{M}$. On the other hand, one could in
fact use more abstract generalized dynamics as in Section \ref{AfdW2} instead
of a semigroup $\alpha$, since the semigroup property will not be used
explicitly. But to clarify how the semigroup property can be spoiled by
reduction, we use the setting where $\alpha$ is a semigroup.
\end{remark}

\begin{remark}
The reason for introducing the augmented systems, is that balance between the
conditional expectations $P_{1\otimes S}^{\mu_{R}}$ will play an essential
role in connecting Wasserstein distances on the composite systems, to those on
the reduced systems discussed below.
\end{remark}

The \emph{reduced system} on $S$, or the \emph{reduction} of $\mathbf{A}$ to
$S$, is defined to be the system
\[
\mathbf{A}^{\mathfrak{r}}=
\left(  S,\alpha^{\mathfrak{r}},\sigma^{\mu_{S}},\mu_{S}\right)  ,
\]
where the reduced dynamics $\alpha^{\mathfrak{r}}:S\rightarrow S$ is given by
\[
\alpha_{g}^{\mathfrak{r}}:=P_{S}^{\mu_{R}}\circ\alpha_{g}\circ\iota_{S,M},
\]
for all $g\in\Gamma$, in terms of
\[
P_{S}^{\mu_{R}}=
\mu_{R}\bar{\otimes}\id_{S}:M\rightarrow S
\]
and
\[
\iota_{S,M}:S\rightarrow M:s\mapsto 1_R\otimes s.
\]
Note that indeed
\[
\mu_{S}\circ\alpha_{g}^{\mathfrak{r}}=
\mu\circ\alpha_{g}\circ\iota_{S,M}=
\mu\circ\iota_{S,M}=
\mu_{S},
\]
as is required of a system. One of the main points here, is that
$\alpha^{\mathfrak{r}}$ clearly need not be a semigroup, despite the fact that
$\alpha$ is, since $\alpha$ need not be product dynamics (due to interaction).
This reduced system $\mathbf{A}^{\mathfrak{r}}$ is what we referred to as
$\mathbf{S}$ in the preliminary discussion at the beginning of the section.

The set of reduced systems obtained in this way, will be denoted by
\[
Y.
\]
The reduced dynamics obtained as part of these reduced systems, can be viewed
as a special class of reduced dynamics, as we assume that they are reduced
from dynamics leaving $\mu_{R}\bar{\otimes}\mu_{S}$ invariant. This is not the
most general situation one could consider, but this class is what fits into
our framework for Wasserstein distances.

Although we are eventually interested only on systems on $M$, a key point
regarding balance (see Proposition \ref{BalOordr} below), and its proof, is
clarified when generalized, by also considering analogous systems on a second
von Neumann algebra $N$. So for $\sigma$-finite von Neumann algebras $K$ and
$L$, we set
\[
N:=K\bar{\otimes}L,
\]
\[
\mathbf{B}=\left(  N,\beta,\sigma^{\nu},\nu\right)  ,
\]
and
\[
\mathbf{B}^{\mathfrak{p}}\mathbf=
\left(  N,\beta,\sigma^{\nu},P_{1\otimes L}^{\nu_{K}},\nu\right)  ,
\]
where $\beta$ is a semigroup of u.c.p. maps leaving
\[
\nu=\nu_{K}\bar{\otimes}\nu_{L}\in\mathfrak{F}(N),
\]
invariant, with $\nu_{K}\in\mathfrak{F}(K)$ and $\nu_{L}\in\mathfrak{F}(L)$,
and the same sets $Z_{1}$, $Z_{2}$ and $Z_{3}$ as before. As mentioned in
Remark \ref{balUitbr}, the notion of balance $\mathbf{A}\omega\mathbf{B}$
still makes sense.

In this setup, balance between the augmented composite systems, carries over
to the reduced systems, as will be seen in Proposition \ref{BalOordr} below,
but first a lemma to get there:

\begin{lemma}
\label{redOordPl}Given $\omega\in T(\mu,\nu)$, with $\mu=\mu_{R}\bar{\otimes
}\mu_{S}$ and $\nu=\nu_{K}\bar{\otimes}\nu_{L}$ as above, we set
\[
\omega^{\mathfrak{r}}:=\omega\circ\left(  \iota_{S,M}\odot\iota_{L',N'}\right)  \in T(\mu_{s},\nu_{L}).
\]
Assuming the balance condition
\[
\left(  M,P_{1\otimes S}^{\mu_{R}},\mu\right)  \omega\left(  N,P_{1\otimes
L}^{\nu_{K}},\nu\right)  ,
\]
it follows that
\[
E_{\omega^{\mathfrak{r}}}=P_{L}^{\nu_{K}}\circ E_{\omega}\circ\iota_{S,M}.
\]
%
\end{lemma}

\begin{proof}
From $\omega\in T(\mu,\nu)$ it is easily checked that $\omega^{\mathfrak{r}
}\in T\left(  \mu_{S},\nu_{L}\right)  $.
%
Since, according to (\ref{E}), $E_{\omega^{\mathfrak{r}}}:S\rightarrow L$ is
determined by $\omega^{\mathfrak{r}}=\delta_{\nu_{L}}\circ\left(
E_{\omega^{\mathfrak{r}}}\odot\id_{L'}\right)  $, we
can verify this lemma as follows using (\ref{BalE}): for all $s\in S$ and
$l'\in L'$,
\begin{align*}
\omega^{\mathfrak{r}}(s\otimes l')  &  =\omega\left(  \iota
_{S,M}(s)\otimes\iota_{L',N'}(l')\right) \\
&  =\delta_{\nu}\left(  E_{\omega}(1_{R}\otimes s)\otimes\iota_{L',N'}(l')\right) \\
&  =\delta_{\nu}\left(  E_{\omega}\circ P_{1\otimes S}^{\mu_{R}}(1_{R}\otimes
s)\otimes\iota_{L',N'}(l')\right) \\
&  =\delta_{\nu}\left(  P_{1\otimes L}^{\nu_{K}}\circ E_{\omega}(\iota
_{S,M}(s))\otimes\iota_{L',N'}(l')\right) \\
&  =\delta_{\nu_{K}}\odot\delta_{\nu_{L}}\left(  \left(  1_{K}\otimes
P_{L}^{\nu_{K}}\circ E_{\omega}\circ\iota_{S,M}(s)\right)  \otimes
(1_{K'}\otimes l')\right) \\
&  =\delta_{\nu_{L}}\left(  P_{L}^{\nu_{K}}\circ E_{\omega}\circ\iota
_{S,M}(s)\otimes l'\right)  ,
\end{align*}
as required.
\end{proof}

I.e., the u.c.p. map $E_{\omega^\mathfrak{r}}$ corresponding to the \emph{reduced}
transport plan $\omega^\mathfrak{r}$, is given by the \emph{reduction} of $E_{\omega}$.
Now, as promised, balance between the augmented composite systems, carries
over to the reduced systems:

\begin{proposition}
\label{BalOordr}If $\mathbf{A}^{\mathfrak{p}}\omega\mathbf{B}^{\mathfrak{p}}
$, then $\mathbf{A}^{\mathfrak{r}}\omega^{\mathfrak{r}}\mathbf{B}^{\mathfrak{r}}$.
%
\end{proposition}

\begin{proof}
Using Lemma \ref{redOordPl}, the given balance conditions, as well as
elementary relations of the form $\iota_{L,N}\circ P_{L}^{\nu_{K}}=P_{1\otimes
L}^{\nu_{K}}$, $P_{L}^{\nu_{K}}\circ P_{1\otimes L}^{\nu_{K}}=P_{L}^{\nu_{K}}$
and $P_{1\otimes S}^{\mu_{R}}\circ\iota_{S,M}=\iota_{S,M}$, it follows that
%
\begin{align*}
\beta_{g}^{\mathfrak{r}}\circ E_{\omega^{\mathfrak{r}}}  &  =P_{L}^{\nu_{K}
}\circ\beta_{g}\circ\iota_{L,N}\circ P_{L}^{\nu_{K}}\circ E_{\omega}\circ
\iota_{S,M}\\
&  =P_{L}^{\nu_{K}}\circ P_{1\otimes L}^{\nu_{K}}\circ\beta_{g}\circ
P_{1\otimes L}^{\nu_{K}}\circ E_{\omega}\circ\iota_{S,M}\\
&  =P_{L}^{\nu_{K}}\circ E_{\omega}\circ P_{1\otimes S}^{\mu_{R}}\circ
\alpha_{g}\circ P_{1\otimes S}^{\mu_{R}}\circ\iota_{S,M}\\
&  =P_{L}^{\nu_{K}}\circ E_{\omega}\circ\iota_{S,M}\circ P_{S}^{\mu_{R}}
\circ\alpha_{g}\circ\iota_{S,M}\\
&  =E_{\omega^{\mathfrak{r}}}\circ\alpha_{g}^{\mathfrak{r}}.
\end{align*}
Furthermore, since $\sigma_{t}^{\mu}=\sigma_{t}^{\mu_{R}}\bar{\otimes}
\sigma_{t}^{\mu_{S}}$ and $\sigma_{t}^{\nu}=\sigma_{t}^{\nu_{K}}\bar{\otimes
}\sigma_{t}^{\nu_{L}}$,%
\begin{align*}
\sigma_{t}^{\nu_{L}}\circ E_{\omega^{\mathfrak{r}}}  &  =P_{L}^{\nu_{K}}
\circ\sigma_{t}^{\nu}\circ E_{\omega}\circ\iota_{S,M}\\
&  =P_{L}^{\nu_{K}}\circ E_{\omega}\circ\sigma_{t}^{\mu}\circ\iota_{S,M}\\
&  =P_{L}^{\nu_{K}}\circ E_{\omega}\circ\iota_{S,M}\circ\sigma_{t}^{\mu_{S}}\\
&  =E_{\omega^{\mathfrak{r}}}\circ\sigma_{t}^{\mu_{S}},
\end{align*}
which completes the proof by (\ref{BalE}).
\end{proof}

In the remainder of the section, we proceed with our main interest, namely the
case
\[
M=N\text{, }R=K\text{ and }S=L.
\]

When using $T_{\sigma}(\mathbf{A},\mathbf{B})$ in Definition \ref{T(A,B)}, we
need to apply Proposition \ref{BalOordr} to the KMS-duals of the systems as
well, which leads us to the next proposition.

\begin{proposition}
\label{KMSkomMet_p&r}For every $\mathbf{A}\in X_{\otimes}$, we have $\left(
\mathbf{A}^{\mathfrak{p}}\right)  ^{\sigma}=\left(  \mathbf{A}^{\sigma
}\right)  ^{\mathfrak{p}}$ and $\left(  \mathbf{A}^{\mathfrak{r}}\right)
^{\sigma}=\left(  \mathbf{A}^{\sigma}\right)  ^{\mathfrak{r}}$.
%
\end{proposition}

\begin{proof}
To show $\left(  \mathbf{A}^{\mathfrak{p}}\right)  ^{\sigma}=\left(
\mathbf{A}^{\sigma}\right)  ^{\mathfrak{p}}$, we clearly just have to check
that $\left(  P_{1\otimes S}^{\mu_{R}}\right)  ^{\sigma}=P_{1\otimes S}
^{\mu_{R}}$, i.e., that $P_{1\otimes S}^{\mu_{R}}$ is KMS-symmetric (w.r.t.
$\mu$), since the remaining conditions are trivially satisfied. First note
that according to Definition \ref{KMS-duaal},
\[
\left(  P_{1\otimes S}^{\mu_{R}}\right)  '=P_{1\otimes S'
}^{\mu_{R}'},
\]
since for all $r\in R$, $r'\in R'$, $s\in S$ and $s'\in S'$,
\begin{align*}
&  \left\langle \Lambda_{\mu},(r\otimes s)\left(  (\mu_{R}'1_{R'})\otimes\id_{S'}(r'\otimes s')\right)  \Lambda_{\mu}\right\rangle \\
&  =\mu_{R}(r)\mu_{R}'(r')\left\langle \Lambda_{\mu_{S}
},ss'\Lambda_{\mu_{S}}\right\rangle \\
&  =\left\langle \Lambda_{\mu},\left(  (\mu_{R}1_{R})\otimes\id
_{S}(r\otimes s)\right)  (r'\otimes s')\Lambda_{\mu
}\right\rangle ,
\end{align*}
simply because $\Lambda_{\mu}=\Lambda_{\mu_{R}}\otimes\Lambda_{\mu_{S}}$. We
therefore have
\begin{align*}
\left(  P_{1\otimes S}^{\mu_{R}}\right)  ^{\sigma}  &  =j_{\mu}\circ
P_{1\otimes S'}^{\mu_{R}'}\circ j_{\mu}\\
&  =\left(  j_{\mu_{R}}\bar{\otimes}j_{\mu_{S}}\right)  \circ\left(  \left(
\mu_{R}'1_{R'}\right)  \bar{\otimes}\id
_{S'}\right)  \circ\left(  j_{\mu_{R}}\bar{\otimes}j_{\mu
_{S}}\right) \\
&  =\left[  \left(  \mu_{R}'\circ j_{\mu_{R}}\right)  j_{\mu_{R}
}(1_{R'})\right]  \bar{\otimes}\left(  j_{\mu_{S}}\circ
\id_{S'}\circ j_{\mu_{S}}\right) \\
&  =\left(  \mu_{R}1_{R}\right)  \bar{\otimes}\id
_{S}\\
&  =P_{1\otimes S}^{\mu_{R}}.
\end{align*}

Next we show $\left(  \mathbf{A}^{\mathfrak{r}}\right)  ^{\sigma}=\left(
\mathbf{A}^{\sigma}\right)  ^{\mathfrak{r}}$. In this case the only
non-trivial part is to check that $\left(  \alpha^{\mathfrak{r}}\right)
^{\sigma}=\left(  \alpha^{\sigma}\right)  ^{\mathfrak{r}}$. We have
\[
\left(  \alpha^{\mathfrak{r}}\right)  ^{\sigma}=\left(  P_{S}^{\mu_{R}}
\circ\alpha_{g}\circ\iota_{S,M}\right)  ^{\sigma}=\iota_{S,M}^{\sigma}
\circ\alpha_{g}^{\sigma}\circ\left(  P_{S}^{\mu_{R}}\right)  ^{\sigma}
\]
and
\[
\left(  \alpha^{\sigma}\right)  ^{\mathfrak{r}}=P_{S}^{\mu_{R}}\circ\alpha
_{g}^{\sigma}\circ\iota_{S,M}.
\]
To determine $\iota_{S,A}^{\sigma}:M\rightarrow S$, we calculate
\begin{align*}
\left\langle \Lambda_{\mu},\iota_{S,M}(s)\left(  r'\otimes s'\right)  \Lambda_{\mu}\right\rangle  &  =\left\langle \Lambda_{\mu},\left(
r'\otimes\left(  ss'\right)  \right)  \Lambda_{\mu
}\right\rangle \\
&  =\left\langle \Lambda_{\mu_{R}},r'\Lambda_{\mu_{R}}\right\rangle
\left\langle \Lambda_{\mu_{S}},ss'\Lambda_{\mu_{S}}\right\rangle \\
&  =\left\langle \Lambda_{\mu_{S}},s\mu_{R}'(r')s'\Lambda_{\mu_{S}}\right\rangle \\
&  =\left\langle \Lambda_{\mu_{S}},s\mu_{R}'\bar{\otimes
}\id_{S'}(r'\otimes s'
)\Lambda_{\mu_{S}}\right\rangle ,
\end{align*}
which by $\sigma$-weak continuity of $\mu_{R}'\bar{\otimes
}\id_{S'}$ (being the tensor product of normal
maps), along with Definition \ref{KMS-duaal}, means that
\[
\iota_{S,M}'=\mu_{R}'\bar{\otimes}\id
_{S'}=P_{S'}^{\mu_{R}'}:M'\rightarrow
S',
\]
which in turns tells us that
\[
\iota_{S,M}^{\sigma}=j_{\mu_{s}}\circ P_{S'}^{\mu_{R}'}\circ
j_{\mu}=P_{S}^{\mu_{R}},
\]
similar to the calculation for $\left(  P_{1\otimes S}^{\mu_{R}}\right)
^{\sigma}$ above. Thus, $\left(  P_{S}^{\mu_{R}}\right)  ^{\sigma}=\left(
\iota_{S,M}^{\sigma}\right)  ^{\sigma}=\iota_{S,M}$, hence $\left(
\alpha^{\mathfrak{r}}\right)  ^{\sigma}=\left(  \alpha^{\sigma}\right)
^{\mathfrak{r}}$.
\end{proof}

The next corollary, which now follows from Remark \ref{TvsTsig}, is not needed
in this section, but becomes relevant in the next. It implies that under the
given conditions, the modular balance conditions for reduced systems, 
reduce to balance between the modular dynamics.

\begin{corollary}
\label{redKMSvirGroep}If $\Gamma$ is a group and $\alpha$ a group
representation as $\ast$-automorphisms of $M$, then $\left(  \alpha
_{g}^{\mathfrak{r}}\right)  ^{\sigma}=\alpha_{g^{-1}}^{\mathfrak{r}}$ for all
$g\in\Gamma$.
\end{corollary}

To connect the transport cost function for reduced systems to the cost for
composite systems, we need the following:

\begin{lemma}
\label{kosteGelyk}Consider $\mu=\mu_{R}\bar{\otimes}\mu_{S},\nu=\nu_{R}
\bar{\otimes}\nu_{S}\in\mathfrak{F}(M)$ as before. For $\omega\in T\left(
\mu,\nu\right)  $, and assuming the balance condition
\[
\left(  M,P_{1\otimes S}^{\mu_{R}},\mu\right)  \omega\left(  M,P_{1\otimes
S}^{\nu_{R}},\nu\right)  ,
\]
it follows that
\[
\nu\left(  k^{\ast}E_{\omega}(k)\right)  =\nu_{S}\left(  w^{\ast}
E_{\omega^{\mathfrak{r}}}(w)\right)
\]
for any $w\in S$ and $k:=1\otimes w\in M$.
%
\end{lemma}

\begin{proof}
Note that for $r\in R$ and $s_{1},s_{2}\in S$, one has
\[
P_{S}^{\mu_{R}}\left(  (1\otimes s_{1})(r\otimes s_{2})\right)  =\mu
_{R}(r)s_{1}s_{2}=s_{1}P_{S}^{\mu_{R}}\left(  r\otimes s_{2}\right)  .
\]
Since $P_{1\otimes S}^{\mu_{R}}$ is normal (it is the tensor product of normal
maps) and multiplication of operators is $\sigma$-weakly continuous in each
factor separately, we know that $P_{S}^{\mu_{R}}\left(  (1\otimes
s)(\cdot)\right)  :M\rightarrow S$ is normal. Hence
\[
P_{S}^{\mu_{R}}\left(  \iota_{S,M}(s)a\right)  =sP_{S}^{\mu_{R}}\left(
a\right)
\]
for all $a\in M$ and $s\in S$. Consequently
\[
\nu\left(  \iota_{S,M}(s)a\right)  =\nu_{S}\circ P_{S}^{\nu_{R}}\left(
\iota_{S,M}(s)a\right)  =\nu_{S}\left(  sP_{S}^{\nu_{R}}\left(  a\right)
\right)  .
\]

In particular, because of the balance assumption,%
\[
\nu\left(  \iota_{S,M}(s)E_{\omega}\circ\iota_{S,M}(w)\right)  =\nu_{S}\left(
sP_{S}^{\nu_{R}}\circ E_{\omega}\circ\iota_{S,M}(w)\right)  =\nu_{S}\left(
sE_{\omega^{\mathfrak{r}}}(w)\right)
\]
by Lemma \ref{redOordPl}, which for the special case $s=w^{\ast}$ gives
$\nu\left(  k^{\ast}E_{\omega}(k)\right)  =\nu_{S}\left(  w^{\ast}
E_{\omega^{\mathfrak{r}}}(w)\right)  $.
\end{proof}

The ideas of this section can now be brought together by the next definition
and result, which show how Wasserstein distances on the set $Y$ of reduced systems
relate to Wasserstein distances on the set $X_{\otimes}$ of composite systems. Note that in
the following definition we restrict the modular Wasserstein pseudometrics
given by Theorem \ref{metriek}, to $X_{\otimes}$ and $X_{\otimes
}^{\mathfrak{p}}$ respectively.
%

\begin{definition}
\label{redW2}In terms of the modular Wasserstein pseudometric $W_{\sigma}$ on
$X_{\otimes}^{\mathfrak{p}}$ associated to given $k_{1},...,k_{n}\in M$, the
corresponding \emph{reduced} modular Wasserstein pseudometric $W_{\sigma
}^{\mathfrak{r}}$ on $X_{\otimes}$, associated to $k_{1},...,k_{n}$, is
defined by
\[
W_{\sigma}^{\mathfrak{r}}\left(  \mathbf{A},\mathbf{B}\right)  =W_{\sigma
}\left(  \mathbf{A}^{\mathfrak{p}},\mathbf{B}^{\mathfrak{p}}\right)
\]
for all $\mathbf{A},\mathbf{B}\in X_{\otimes}$.
\end{definition}

\begin{theorem}
Consider the modular Wasserstein pseudometric $W_{\sigma}$ on $Y$ associated
to $w_{1},...,w_{n}\in S$, and the reduced modular Wasserstein pseudometric
$W_{\sigma}^{\mathfrak{r}}$ on $X_{\otimes}$ associated to $k_{1},...,k_{n}\in
M$ given by
\[
k_{l}:=1_{R}\otimes w_{l}
\]
for $l=1,...,n$. Then
\[
W_{\sigma}\left(  \mathbf{A}^{\mathfrak{r}},\mathbf{B}^{\mathfrak{r}}\right)
\leq W_{\sigma}^{\mathfrak{r}}\left(  \mathbf{A},\mathbf{B}\right)
\]
for all $\mathbf{A},\mathbf{B}\in X_{\otimes}$.
\end{theorem}

\begin{proof}
For any $\omega\in T\left(  \mathbf{A}^{\mathfrak{p}},\mathbf{B}
^{\mathfrak{p}}\right)  $, and writing $w=(w_{1},...,w_{n})$ and
$k=(k_{1},...,k_{n})$, we have
\[
I_{w}(\omega^{\mathfrak{r}})=I_{k}(\omega)
\]
in terms of (\ref{Ialt}), because of Lemma \ref{kosteGelyk} and $\omega
^{\mathfrak{r}}\in T(\mu_{s},\nu_{S})$.

In addition, $\mathbf{A}^{\mathfrak{r}}\omega^{\mathfrak{r}}\mathbf{B}
^{\mathfrak{r}}$ by Proposition \ref{BalOordr}. Assuming $\left(
\mathbf{A}^{\mathfrak{p}}\right)  ^{\sigma}\omega\left(  \mathbf{B}
^{\mathfrak{p}}\right)  ^{\sigma}$ as well, we know from Proposition
\ref{KMSkomMet_p&r} that $\left(  \mathbf{A}^{\sigma}\right)  ^{\mathfrak{p}
}\omega\left(  \mathbf{B}^{\sigma}\right)  ^{\mathfrak{p}}$, hence $\left(
\mathbf{A}^{\sigma}\right)  ^{\mathfrak{r}}\omega^{\mathfrak{r}}\left(
\mathbf{B}^{\sigma}\right)  ^{\mathfrak{r}}$, giving $\left(  \mathbf{A}
^{\mathfrak{r}}\right)  ^{\sigma}\omega^{\mathfrak{r}}\left(  \mathbf{B}
^{\mathfrak{r}}\right)  ^{\sigma}$, by Propositions \ref{BalOordr} and
\ref{KMSkomMet_p&r}. Recalling Remark \ref{TvsTsig}, this means that
\[
\left\{  \omega^{\mathfrak{r}}:\omega\in T_{\sigma}(\mathbf{A}^{\mathfrak{p}
},\mathbf{B}^{\mathfrak{p}})\right\}  \subset T_{\sigma}(\mathbf{A}
^{\mathfrak{r}},\mathbf{B}^{\mathfrak{r}}).
\]

In view of Definition \ref{redW2}, these two facts imply the result.
\end{proof}

Of course, assuming that $\left\{  w_{1},...,w_{n}\right\}  ^{\ast}=\left\{
w_{1},...,w_{n}\right\}  $, and that this set generates $S$, the associated
$W_{\sigma}$ on $Y$ in this theorem is in fact a metric according to Theorem
\ref{metriek}.

An analogous result is achieved for $W$ instead of $W_{\sigma}$. Simply
consider the set $\mathsf{X}_{\otimes}$ of systems $\mathbf{A}=\left(
M,\alpha,\mu\right)  $ of the form (\ref{basVorm}), but with $\sigma^{\mu}$
dropped, and $\mathsf{X}_{\otimes}^{\mathfrak{p}}$ consisting of the
corresponding augmented systems $\mathbf{A}^{\mathfrak{p}}=\left(
M,\alpha,P_{1\otimes S}^{\mu_{R}},\mu\right)  $. Similarly, define $\mathsf{Y}$
as the set of corresponding reduced systems $\mathbf{A}^{\mathfrak{r}}=\left(
S,\alpha^{\mathfrak{r}},\mu_{S}\right)  $. Define the \emph{reduced}
asymmetric Wasserstein pseudometric $W^{\mathfrak{r}}\left(  \mathbf{A}
,\mathbf{B}\right)  :=W\left(  \mathbf{A}^{\mathfrak{p}},\mathbf{B}
^{\mathfrak{p}}\right)  $ for all $\mathbf{A},\mathbf{B}\in\mathsf{X}
_{\otimes}$, from the asymmetric Wasserstein pseudometric $W$ on
$\mathsf{X}_{\otimes}^{\mathfrak{p}}$ associated to $k_{1},...,k_{n}\in M$,
given by Theorem \ref{asim}. Remove aspects related to KMS-duals from the
preceding proof. Then we analogously have the following result:

\begin{theorem}
Consider the modular Wasserstein pseudometric $W$ on $\mathsf{Y}$ associated
to $w_{1},...,w_{n}\in S$, and the reduced modular Wasserstein pseudometric
$W^{\mathfrak{r}}$ on $\mathsf{X}_{\otimes}$ associated to $k_{1},...,k_{n}\in
M$ given by
\[
k_{l}:=1_{R}\otimes w_{l}
\]
for $l=1,...,n$. Then
\[
W\left(  \mathbf{A}^{\mathfrak{r}},\mathbf{B}^{\mathfrak{r}}\right)  \leq
W^{\mathfrak{r}}\left(  \mathbf{A},\mathbf{B}\right)
\]
for all $\mathbf{A},\mathbf{B}\in\mathsf{X}_{\otimes}$.
\end{theorem}


\section{Examples\label{AfdVbe}}

In this section we present a few simple but enlightening examples to provide
some insight into the Wasserstein distances. As Wasserstein distances are
usually defined on the states in the classical case (i.e., on probability
measures), we are instead particularly interested in the behaviour of the
Wasserstein metrics of this paper with regards to the dynamics, including the
modular dynamics.

We investigate how sharply these distances distinguish different dynamics,
given the same pair of states. As will be seen, not necessarily very sharply,
indicating that Wasserstein distance between systems chiefly measures
distances between the states of systems, but in a way that is regulated by the
dynamics. This may be expected from the definition. An important related point
is that balance conditions can cause jumps in values of the Wasserstein
distances. In particular this can happen for modular balance, indicating that
the latter condition may be unnatural in some ways when we are interested in
distances between the states themselves, not between systems. On the other
hand, it will be argued that these jumps should be of value in discerning
qualitative differences between the dynamics of two systems.

The other point we make in this section, is that symmetry of a Wasserstein
distance can indeed be lost if balance with respect to the modular dynamics is
not included.

We divide the section into subsections. The first concerns a very simple conceptual
point used in the examples, the last gives a brief summary of the
conclusions and possible implications, while the others consider the examples
in turn.

\subsection{Equal distance}


Given $\Upsilon$ and the $Z_{\upsilon}$'s, let $D_{\mu}$ denote the set of all
$\alpha$'s appearing in Definition \ref{stelsel}. Given $\alpha_{1},\alpha
_{2}\in D_{\mu}$ and $\beta_{1},\beta_{2}\in D_{\nu}$, write
\[
\mathbf{A}_{l} = \left( M,\alpha_{l},\mu \right) 
\text{ and } 
\mathbf{B}_{l} = \left( M,\beta_{l},\nu \right)
\]
for $l=1,2$. 
Note that if
\begin{equation}
T_{\sigma} \left( \mathbf{A}_{1},\mathbf{B}_{1} \right) =
T_{\sigma} \left( \mathbf{A}_{2},\mathbf{B}_{2} \right),
\label{ekwPare}
\end{equation}
then
\[
W_{\sigma}\left(  \mathbf{A}_{1},\mathbf{B}_{1}\right)  =W_{\sigma}\left(
\mathbf{A}_{2},\mathbf{B}_{2}\right)
\]
for any modular Wasserstein pseudometric on $X$. This is not a particularly
powerful condition, but since it does not refer to details regarding the
transport cost function, it can give us general conclusions independent of
transport cost. It can also be comparatively straightforward to check for
simple systems. The analogous remarks are of course also true for $W$.

\subsection{Unitary dynamics in $M_{2}$\label{OAfdUn}}


Consider the $2\times2$ complex matrices $M_{2}$. It is simpler to work
directly in this representation, instead of the standard form $M=M_{2}
\otimes1_{2}$ with $1_{2}$ the $2\times2$ identity matrix. We use
$\Upsilon=\{1\}$, and $Z_{1}=\mathbb{Z}$ in the notation of
Definition \ref{stelsel}, and view the modular dynamics
separately, as we did in Section \ref{AfdW2}. Given a fixed $\theta
\in\mathbb{R}$, we define systems $\mathbf{A}_{\theta}=\left(  \alpha_{\theta
},\mu\right)  $ and $\mathbf{B}_{\theta}=\left(  \alpha_{\theta},\nu\right)  $
on $M_{2}$ by the dynamics
\[
\alpha_{\theta}(a)=U_{\theta}aU_{\theta}^{\ast}\text{ }
\]
acting through iteration (hence $Z_{1}=\mathbb{Z}$ above) and the states
\[
\mu(a)=\Tr(\zeta a)\text{ and }\nu(a)=\Tr(\eta
a),
\]
where $U_{\theta}$ is the unitary matrix
\[
U_{\theta}=\left[
\begin{array}
[c]{cc}%
1 & 0\\
0 & e^{i\theta}
\end{array}
\right]
\]
for every $\theta\in\mathbb{R}$, and $\zeta$ and $\eta$ are the density
matrices%
\[
\zeta=\left[
\begin{array}
[c]{cc}
p_{1} & 0\\
0 & p_{2}
\end{array}
\right]  \text{ and }\eta=\left[
\begin{array}
[c]{cc}
q_{1} & 0\\
0 & q_{2}
\end{array}
\right]
\]
with $0<p_{1},q_{1}<1$ and $p_{1}+p_{2}=q_{1}+q_{2}=1$. Write any $\omega\in
T(\mu,\nu)$ in terms of its $4\times4$ density matrix $\kappa=\left[
\kappa_{kl}\right]  $, that is,
\[
\omega(c)=\Tr(\kappa c)
\]
for all $c\in M_{2}\odot M_{2}$. Note that $M_{2}\odot M_{2}$ is the correct way to
express $M\odot M'$ in our current representation.

We first want to determine the effect of the balance condition $\mathbf{A}
_{\phi}\omega\mathbf{B}_{\theta}$ by itself on $\kappa$, without making any
other assumptions about the matrix $\kappa$. For the moment we therefore assume that
$\kappa=\left[  \kappa_{lm}\right]  $ is an arbitrary $4\times4$ complex
matrix, defining a linear functional $\omega=\Tr(\kappa\cdot)$
on $M_{2}\odot M_{2}$. Then it is elementary to check that the restrictions
placed on $\kappa$ by the balance condition $\mathbf{A}_{\phi}\omega
\mathbf{B}_{\theta}$, is described by the following:

\begin{enumerate}
\item[(a)] $e^{i\phi}\neq1$ implies $\kappa_{13}=\kappa_{31}=\kappa
_{24}=\kappa_{42}=0$, while $e^{i\phi}=1$ has no implications for $\kappa$.

\item[(b)] $e^{i\theta}\neq1$ implies $\kappa_{12}=\kappa_{21}=\kappa
_{34}=\kappa_{43}=0$, while $e^{i\theta}=1$ has no implications for $\kappa$.

\item[(c)] $e^{i\phi}\neq e^{i\theta}$ implies $\kappa_{14}=\kappa_{41}=0$,
while $e^{i\phi}=e^{i\theta}$ has no implications for $\kappa$.

\item[(d)] $e^{i\phi}\neq e^{-i\theta}$ implies $\kappa_{23}=\kappa_{32}=0$,
while $e^{i\phi}=e^{-i\theta}$ has no implications for $\kappa$.
\end{enumerate}

Since $\alpha_{\theta}$ is a $\ast$-automorphism, $\mathbf{A}_{\phi}^{\sigma
}\omega\mathbf{B}_{\theta}^{\sigma}$ is automatically satisfied if
$\mathbf{A}_{\phi}\omega\mathbf{B}_{\theta}$ holds; see Remark \ref{TvsTsig}.

The effect of modular balance $(\sigma^{\mu},\mu)\omega(\sigma^{\nu},\nu)$ is
of exactly the same form. We have
\[
\sigma_{t}^{\mu}(a)=\zeta^{it}a\zeta^{-it}=\left[
\begin{array}
[c]{cc}
a_{11}                                  & \left(\frac{p_2}{p_1}\right)^{-it}a_{12}\\
\left(\frac{p_2}{p_1}\right)^{it}a_{21} & a_{22}
\end{array}
\right],
\]
as opposed to
\[
\alpha_{\theta}(a)=
\left[
\begin{array}
[c]{cc}
a_{11}          & e^{-i\theta}a_{12}\\
e^{i\theta}a_{21} & a_{22}
\end{array}
\right]  ,
\]
in terms of $a=[a_{lm}]$. It follows that modular balance has the same set of
implications for $\kappa$ as (a) to (d) above, with the corresponding
inequality in each condition just having to hold for some value of $t$.
Consequently, $e^{i\phi}$ can be replaced by $p_{2}/p_{1}$, and $e^{i\theta}$
by $q_{2}/q_{1}$, in (a) to (d). I.e., $p_{2}/p_{1}\neq1$, implies that
($p_{2}/p_{1})^{it}\neq1$ for some $t$, etc.

Consider the case where the inequality in each condition (a) to (d) for
modular dynamics holds, namely $p_{2}/p_{1}\neq1$, $q_{2}/q_{1}\neq1$,
$p_{2}/p_{1}\neq q_{2}/q_{1}$ and $p_{2}/p_{1}\neq q_{1}/q_{2}$. We intend to
show that, irrespective of the transport cost function, the modular
Wasserstein distance
\[
W_{\sigma}(\mathbf{A}_{\phi},\mathbf{B}_{\theta})
\]
is independent of $\phi$ and $\theta$. To see this, note that in this case,
$\kappa$ is diagonal by (a) to (d) applied to modular balance. Therefore
$\mathbf{A}_{\phi}\omega\mathbf{B}_{\theta}$ (and thus necessarily
$\mathbf{A}_{\phi}^{\sigma}\omega\mathbf{B}_{\theta}^{\sigma}$) has no further
effect on $\kappa$. To this we need to add the coupling property, which in our
current representation reads 
$\omega(a\otimes1)=\mu(a)$ and $\omega(1\otimes a)=\nu(a)$, 
and the positivity of $\kappa$, neither of which depend on $\phi$
and $\theta$. It follows that 
$T_{\sigma}\left(  \mathbf{A}_{\phi},\mathbf{B}_{\theta}\right) $, 
and thus $W_{\sigma}(\mathbf{A}_{\phi},\mathbf{B}_{\theta})$, 
are independent of $\phi$ and $\theta$, as claimed.
Since $\mu\neq\nu$, as $p_{2}/p_{1}\neq q_{2}/q_{1}$,
we of course have $W_{\sigma}(\mathbf{A}_{\phi},\mathbf{B}_{\theta})\neq0$.

Other cases can be similarly explored, with similar conclusions. For example,
assuming that $e^{i\phi}\neq1$, $e^{i\theta}\neq1$, $e^{i\phi}\neq e^{i\theta}$ 
and $e^{i\phi}\neq e^{-i\theta}$, but dropping the assumptions 
$p_{2}/p_{1}\neq1$, $q_{2}/q_{1}\neq1$, $p_{2}/p_{1}\neq q_{2}/q_{1}$ and
$p_{2}/p_{1}\neq q_{1}/q_{2}$, 
we obtain diagonal $\kappa$ from $\mathbf{A}_\phi\omega\mathbf{B}_\theta$. In this situation modular balance obviously
holds no further implications for $\kappa$, since (a) to (d) will at most
reproduce the conditions already covered by 
$\mathbf{A}_{\phi}\omega\mathbf{B}_{\theta}$. It follows that 
$W(\mathbf{A}_\phi,\mathbf{B}_\theta) =
W_\sigma(\mathbf{A}_\phi,\mathbf{B}_\theta)
\neq0$, 
and that the value is independent of $\phi$ and $\theta$. 
It is non-zero, even for $\mu=\nu$,
simply because $\alpha_{\phi}\neq\alpha_{\theta}$.

\subsection{Asymmetry of $W$\label{OAfdAsim}}


We proceed with the previous example. Here and in the next subsection we
calculate some values of $W$ and $W_{\sigma}$ explicitly for a specific
transport cost function. In particular, we use this example to show that
without the modular balance conditions in Definition \ref{T(A,B)}, symmetry is
lost. That is, $W$ in Theorem \ref{asim}, can indeed lack symmetry. In this
example, the lack of symmetry is specifically due to an absence of balance
between the modular dynamics, since the KMS duals are in balance with respect to all
the allowed transport plans, as mentioned above.

We now calculate $W(\mathbf{A}_{\theta},\mathbf{B}_{\theta})$. For simplicity
we assume that
\[
0<p\leq q<\frac{1}{2},
\]
where we have written $p=p_{1}$ and $q=q_{1}$.

The transport cost function used in defining $W$, is chosen to be given by the
following two self-adjoint generators for $M_{2}$:
\[
k_{1}=\left[
\begin{array}
[c]{cc}
0 & 0\\
0 & 1
\end{array}
\right]  \text{ and }k_{2}=\left[
\begin{array}
[c]{cc}%
0 & 1\\
1 & 0
\end{array}
\right]  .
\]
Then,
assuming $e^{i\theta}\neq1$ and $e^{i\theta}\neq e^{-i\phi}$,
\[
W(\mathbf{A}_{\theta},\mathbf{B}_{\theta})=2+q-p-2\sqrt{\frac{p}{q}}
\]
and%
\[
W(\mathbf{B}_{\theta},\mathbf{A}_{\theta})=2+q-p-2\sqrt{\frac{1-q}
{1-p}}.
\]
This is achieved by determining the set of linear functionals $\omega$
satisfying balance, restricting it further by the coupling property and then
positivity, to obtain the set of all allowed transport plans, and finally
using standard calculus to obtain the minimum cost. One way of doing this, is
to use (\ref{Ialt}) directly. Another way, is to use a ``cost matrix'' $c_{k}$
as given by \cite[Equation (2)]{D}, and using the formula%
\[
I_{k}(\omega)=\omega(c_{k}),
\]
which in finite dimensions is equivalent to (\ref{Ialt}), as can be seen from
\cite[Section 3]{D}.

In particular, we see that%
\[
W(\mathbf{A}_{\theta},\mathbf{B}_{\theta})>W(\mathbf{B}_{\theta}
,\mathbf{A}_{\theta})
\]
when $p<q$. That is, $W$ is not symmetric.

\subsection{Jumps in $W_{\sigma}$\label{OAfdSpronge}}

Including modular balance in Subsection \ref{OAfdAsim} and recalling that
$\mathbf{A}_{\phi}^{\sigma}\omega\mathbf{B}_{\theta}^{\sigma}$ is
automatically satisfied, one finds that because of condition (c) above,
\begin{equation}
W_{\sigma}(\mathbf{A}_{\theta},\mathbf{B}_{\theta})=\left\{
\begin{array}
[c]{l}
2+q-p\text{ if }p<q\\
0\text{ if }p=q.
\end{array}
\right.  \label{WsigVb}
\end{equation}
Note the jump in value from $p<q$ to $p=q$. The reason this happens is as
follows: The condition $\mathbf{A}_{\theta}\omega\mathbf{B}_{\theta}$ forces
$\kappa$ to be diagonal, except for $\kappa_{14}$ and $\kappa_{41}$, which are
not forced to be zero. On the other hand, for $p<q$, modular balance does
force $\kappa$ to be diagonal (so $\kappa_{14}=\kappa_{41}=0$), whereas for
$p=q$, it does not. Thus, for $p<q$, the allowed set of transport plans is
smaller, leading to the larger value in (\ref{WsigVb}) for $W_{\sigma}$,
compared to $W$. But when $p=q$, modular balance does not add further
restrictions on the allowed set of transport plans beyond those due to
$\mathbf{A}_{\theta}\omega\mathbf{B}_{\theta}$, hence in this case we have
$W_{\sigma}(\mathbf{A}_{\theta},\mathbf{B}_{\theta})=W(\mathbf{A}_{\theta
},\mathbf{B}_{\theta})=W(\mathbf{B}_{\theta},\mathbf{A}_{\theta})=0$. In
short, we have two different sets of allowed transport plans for the cases
$p<q$ and $p=q$ respectively.

This jump in $W_{\sigma}(\mathbf{A}_{\theta},\mathbf{B}_{\theta})$ indicates
that the inclusion of modular balance may not be natural in all respects if we
want to consider distances between states. For example, one would prefer
$W_{\sigma}(\mathbf{A}_{\theta},\mathbf{B}_{\theta})$ to converge to zero as
$\mu$ and $\nu$ approach one another in a natural way, as in this case where
$p$ and $q$ approach each other in the usual metric on $\mathbb{R}$.

The latter is indeed what happens in the classical analogy of this example,
for a two point metric space and the expectations with respect to the
probability measures on this space being the states on the von Neumann algebra
$\mathbb{C}^{2}$ of functions on it. No jumps as above occur in this case,
since even in a much more general setting, the classical Wasserstein metric on
the states, metrizes the weak topology on the states \cite[Remark
7.13(iii)]{V1}.

This could mean
that for states, the asymmetric Wasserstein metrics in Theorem \ref{asim}, are
more natural in the bimodule approach. As mentioned in \cite[Section 7]{D},
this may be sensible, as the direction of transport would then appear to be
reflected to some extent in the asymmetric metric.

As the remaining subsections illustrate, jumps can occur in $W$ as well, for
essentially the same reasons, due to other dynamics. For systems (as to
opposed to states) jumps in $W$ and $W_{\sigma}$ could be viewed as feature,
not a drawback, as will be explained.


\subsection{Reduced systems on $M_{2}$\label{OAfdRed}}

We proceed in a similar vein as the previous subsections, but in the context of Section \ref{AfdRed}.

We set $R=S=M_{2}$ in the notation of Section \ref{AfdRed}, but in terms of
the same approach to representations as in Subsection \ref{OAfdUn}. Consider
dynamics, including an interaction between the reservoir and the open system,
described by a Hamiltonian of the form%
\[
h=\Theta\otimes1_{2}+1_{2}\otimes\Phi+\lambda u\otimes v,
\]
where $\Theta,\Phi,u,v\in M_{2}$ are diagonal real matrices, and $\lambda
\geq0$ controls the strength of the interaction.

Let the states $\mu_{R}$ and $\mu_{S}$ be given by diagonal density matrices.
The dynamics on $R\bar{\otimes}S=R\odot S$ is given by 
$\alpha_{t} = e^{iht}(\cdot)e^{-iht}$ for all $t\in\mathbb{R}$, 
which clearly leaves $\mu$ invariant as required. Using
\[
d^{\mu}=\left[
\begin{array}
[c]{rr}%
d_{1}^{\mu} & 0\\
0 & d_{2}^{\mu}%
\end{array}
\right]
\]
with $0<d_{1}^{\mu}<1$ as the density matrix of $\mu_{R}$, the reduction of
$\alpha$ to $S$ gives%
\[
\alpha_{t}^{\lambda}(s):=\alpha_{t}^{\mathfrak{r}}(s)=\left[
\begin{array}
[c]{cc}%
s_{11} & \Xi_{\mu,t}(\lambda)s_{12}\\
\Xi_{\mu,t}^{\ast}(\lambda)s_{21} & s_{22}
\end{array}
\right]
\]
for all $s=\left[  s_{lm}\right]  \in M_{2}$, in terms of
\begin{align*}
\Xi_{\mu,t}(\lambda)  
&  
:= d_{1}^{\mu}e^{i(h_{1}-h_{2})t}+d_{2}^{\mu}e^{i(h_{3}-h_{4})t}\\
&  
= \left[ d_1^\mu e^{i\lambda u_1(v_1-v_2)t} + 
d_2^\mu e^{i\lambda u_2(v_1-v_2)t} \right]  
e^{i(\Phi_{1}-\Phi_{2})t}
\end{align*}
and its complex conjugate $\Xi_{\mu,t}^{\ast}$, where $h_{1},...,h_{4}$,
$\Phi_{1},\Phi_{2}$, $u_{1},u_{2}$ and $v_{1},v_{2}$ are the (diagonal)
entries of $h$, $\Phi$, $u$ and $v$ respectively. Note that $\alpha^{\lambda}$
depends on $\mu_{R}$, but not on the Hamiltonian $\Theta$ used for $R$, or on
$\mu_{S}$. We have thus obtained a reduced system
\[
\mathbf{A}^{\lambda}=\left(  \alpha^{\lambda},\mu_{S}\right)
\]
on $S$, with $\alpha^\lambda$ independent of the state $\mu_S$. Note that
$\alpha^\lambda$ does not have the semigroup property, unless 
$\lambda(u_l-u_2)(v_1-v_2)=0$. 
The latter condition boils down to no interaction.

Following exactly the same procedure, we can obtain another reduced system on
$S$,
\[
\mathbf{B}^{\lambda}=\left(  \beta^{\lambda},\nu_{S}\right),
\]
by using another state $\nu=\nu_{R}\otimes\nu_{S}$ given by diagonal density
matrices, while keeping the rest of the specifications the same.

Using the corresponding restrictions provided by (a) to (d) in Subsection
\ref{OAfdUn}, we can draw analogous conclusions. For example, assume that
$\Phi_{1}\neq\Phi_{2}$, $u_{1},u_{2}>0$, $u_{1}\neq u_{2}$ and $v_{1}\neq
v_{2}$. For the case $\mu_{R}\neq\nu_{R}$ it then follows that
\[
W\left(  \mathbf{A}^{\lambda_{1}},\mathbf{B}^{\lambda_{2}}\right)
\]
is independent of the two parameters $\lambda_{1}$ and $\lambda_{2}$, as long
as at least one of them is not zero. If $\mu_{R}=\nu_{R}$, on the other hand,
then $W\left(  \mathbf{A}^{\lambda_{1}},\mathbf{B}^{\lambda_{2}}\right)  $ is
independent of $\lambda_{1}$ and $\lambda_{2}$, when $\lambda_{1}
\neq\lambda_{2}$, even if one of them is zero, and alternatively when $\lambda
_{1}=\lambda_{2}$. Both these cases hold simply because the allowed $\kappa$
are diagonal when at least one of $\lambda_{1}$ or $\lambda_{2}$ is not zero
(respectively, when $\lambda_{1}\neq\lambda_{2}$), whereas $\kappa_{14}$ and
$\kappa_{41}$ are allowed to be non-zero otherwise. 

Moreover, using the same
transport cost as in Subsection \ref{OAfdAsim}, and $\zeta$ and $\eta$
appearing in Subsection \ref{OAfdUn} as the density matrix of $\mu_{S}$ and
$\nu_{S}$ respectively, we consequently find similar values for the
Wasserstein distance as found in Subsection \ref{OAfdAsim}. More precisely,
again writing $p=p_{1}$ and $q=q_{1}$, and assuming $0<p\leq q<\frac{1}{2}$,
we find:

For $\mu_{R}\neq\nu_{R}$,
\begin{equation}
W(\mathbf{A}^{\lambda_{1}},\mathbf{B}^{\lambda_{2}})=W(\mathbf{B}^{\lambda
_{2}},\mathbf{A}^{\lambda_{1}})=2+q-p \label{WvirRedDin}
\end{equation}
if at least one of $\lambda_{1}$ or $\lambda_{2}$ is not zero, simply because
in this case for $W$, the same sets of transport plans (one set for each of
the two orders in which the systems appear in $W$) are allowed as for
$W_{\sigma}$ in the first line of (\ref{WsigVb}), even for $p=q$, while

\[
W(\mathbf{A}^{0},\mathbf{B}^{0})=2+q-p-2\sqrt{\frac{p}{q}}
\]
and%
\[
W(\mathbf{B}^{0},\mathbf{A}^{0})=2+q-p-2\sqrt{\frac{1-q}{1-p}}.
\]
Analogous to $W_{\sigma}$ in Subsection \ref{OAfdAsim}, we now see a jump in
$W$ when the point $\lambda_{1}=\lambda_{2}=0$ is reached, since this is where
the sets of allowed transport plans change. The same values are achieved for
$\mu_{R}=\nu_{R}$, by the same arguments, but for the cases $\lambda_{1}
\neq\lambda_{2}$ and $\lambda_{1}=\lambda_{2}$ respectively, i.e.,
$W(\mathbf{A}^{\lambda},\mathbf{B}^{\lambda})=2+q-p-2\sqrt{p/q}$, etc.

In this example it appears that $W$ is not sensitive to the interaction
strength $\lambda$, except that for $\mu_{R}\neq\nu_{R}$ it distinguishes
between the case where non-zero interactions are involved, and the case where
there are no interactions, while for $\mu_{R}=\nu_{R}$ it discerns the
case where the interaction strengths differ, from the case where they do not.

When $\mu_{S}\neq\nu_{S}$, the distance $W_{\sigma}$ distinguishes even less
well between the different cases for the reduced dynamics in this example,
since balance between the modular dynamics is then already enough to force
$\kappa$ to be diagonal. 
Keeping in mind the implications of Corollary \ref{redKMSvirGroep}, it
follows as for (\ref{WvirRedDin}), that 
\[
W_\sigma(\mathbf{A}^{\lambda_1},\mathbf{B}^{\lambda_2}) = 2+q-p
\]
for all $\lambda_{1},\lambda_{2}\geq0$,
whether $\mu_{R}\neq\nu_{R}$ or not. For $\mu_{S}=\nu_{S}$, on the other hand,
balance between the modular dynamics does not force $\kappa_{14}$ and
$\kappa_{41}$ to be zero, hence this balance condition becomes redundant, thus
$W_{\sigma}(\mathbf{A}^{\lambda_{1}},\mathbf{B}^{\lambda_{2}})=W(\mathbf{A}
^{\lambda_{1}},\mathbf{B}^{\lambda_{2}})$ for all $\lambda_{1},\lambda_{2}
\geq0$, returning the various values above, but with $p=q$.

Therefore, when we want to distinguish qualitative
differences between dynamics more sharply, $W$ has the advantage, while $W_{\sigma}$ gives a
coarser view, as is to be expected from (\ref{Wfyner}).

\subsection{Translations on quantum tori}

For our last example we make use of disjointness as it appears in the theory
of joinings. Systems are \emph{disjoint}
exactly when the set of transport plans is
trivial, that is, consists solely of the product coupling. In this situation
Wasserstein distances are simple to compute.

In short, we consider actions of $\mathbb{R}^{2}$ on the von Neumann algebraic
quantum torus (or irrational rotation algebra) $M_{\theta}$ for a given
irrational number $\theta$, represented in standard form on the Hilbert space
$G=L^{2}(\mathbb{T}^{2})$, with respect to the normalized Haar measure on
$\mathbb{T}^{2}$, where $\mathbb{T}^{2}=\mathbb{R}^{2}/\mathbb{Z}^{2}$ is the
classical torus. We have such an action $\alpha^{p,q}$ for every pair of real
numbers $p$ and $q$, given by
\[
\alpha_{s,t}^{p,q}=\tau_{ps,qt}
\]
for all $(s,t)\in\mathbb{R}^{2}$, where $\tau$ is the natural translation on
$M_{\theta}$.

The quantum torus is a standard example in operator algebra, but for clarity
we briefly outline its construction. With every $\theta\in\mathbb{R}$, we
associate the von Neumann algebra $M_{\theta}$ generated by the unitary
operators $u,v\in\mathcal{B}(G)$ defined by
\[
\left(  uf\right)  (x,y):=e^{2\pi ix}f\left(  x,y+\theta/2\right)
\]
and
\[
\left(  vf\right)  \left(  x,y\right)  :=e^{2\pi iy}f\left(  x-\theta
/2,y\right)
\]
for all $f\in G$ and $\left(  x,y\right)  \in\mathbb{T}^{2}$. Then
$\Lambda:=1\in G$ is a cyclic and separating vector for $M_{\theta}$, which
defines a faithful normal trace $\Tr$ on $M_{\theta}$ by
\[
\Tr(a):=\left\langle \Lambda,a\Lambda\right\rangle
\]
for all $a\in M_{\theta}$.

A natural action $\tau$ of $\mathbb{R}^{2}$ on $M_{\theta}$, as $\ast
$-automorphisms which leave $\Tr$ invariant, can be obtained
via the classical translations
\[
T_{s,t}(x,y):=\left(  x+s,y+t\right)  \in\mathbb{T}^{2}
\]
for all $\left(  x,y\right)  \in\mathbb{T}^{2}$, by defining
\[
U_{s,t}:G\rightarrow G:f\mapsto f\circ T_{s,t}%
\]
and setting
\[
\tau_{s,t}(a):=U_{s,t}aU_{s,t}^{\ast}%
\]
for all $a\in M_{\theta}$, and all $\left(  s,t\right)  \in\mathbb{R}^{2}$. We
thus have a system $\mathbf{A}_{p,q}=(M_{\theta},\alpha^{p,q}
,\Tr)$ for any real numbers $p$ and $q$.

According to \cite[Example 3.14]{DSt} and Remark \ref{TvsTsig}, we have
\[
T_{\sigma}(\mathbf{A}_{p,q},\mathbf{A}_{c,d})=T(\mathbf{A}_{p,q}
,\mathbf{A}_{c,d})=\{\Tr\odot\Tr
'\}
\]
for all $p,q,c,d\in\mathbb{R}\backslash\{0\}$ with $p/c$ or $q/d$
irrational. Consequently, for a transport cost function given by any
self-adjoint set of generators of $M_{\theta}$, it follows that $W_{\sigma
}(\mathbf{A}_{p,q},\mathbf{A}_{c,d})=W(\mathbf{A}_{p,q},\mathbf{A}_{c,d}
)\neq0$ is independent of $p,q,c$ and $d$, as long as the mentioned
irrationality condition holds. In particular, these Wasserstein distances do
not depend on how far $(p,q)$ is from $(c,d)$ in $\mathbb{R}^{2}$. As an
obvious example, using $k=(u,v,u^{\ast},v^{\ast})$ in Definition \ref{I}, one
easily calculates $W_{\sigma}(\mathbf{A}_{p,q},\mathbf{A}_{c,d})=W(\mathbf{A}
_{p,q},\mathbf{A}_{c,d})=8$, since $E_{\Tr\odot
\Tr'}=\Tr(\cdot)1_{M_{\theta}}$
because of (\ref{E}).

For $p,q,c$ and $d$ for which the above mentioned irrationality condition does
not hold, disjointness could fail, that is, the set of transport plans could
be larger than $\{\Tr\odot\Tr'\}$, potentially leading to jumps in the Wasserstein distances. In
particular, we of course have $W_{\sigma}(\mathbf{A}_{p,q},\mathbf{A}
_{p,q})=W(\mathbf{A}_{p,q},\mathbf{A}_{p,q})=0$ for any $k$ in Definition
\ref{I}.

\subsection{Summary of conclusions}

We considered systems parametrized by constants appearing in their dynamics or
states. Varying these parameters, gives families of related systems. For pairs
of systems, this includes ranges where the dynamics of the two systems are in
some relation, for example having different interactions with a reservoir or
when the state of the reservoir is the same for both reduced systems. Both $W$
and $W_{\sigma}$ are often independent of such parameters, except for jumps at
the special points. These jumps are the result of changes in the allowed set
of transport plans. Such a change tends to occur in a ``discrete'' way when
the two systems' dynamics reaches some threshold, rather than just because of
any small perturbations in the dynamics.

From this limited range of examples, it appears that $W$ and $W_{\sigma}$, and
indeed also the sets of transport plans $T$ and $T_{\sigma}$, tend to be able
to discern qualitative or structural differences between dynamics, but are
less sensitive to finer details or small perturbations in the dynamics. This
is in line with what may be expected from the definition, which adapts the
definition for states (or measures in the classical case) by still considering
a cost function on transport plans between states, rather than cost directly
related to dynamics, but regulates the allowed sets of transport plans using
the dynamics. We expect this behaviour of $W$ and $W_{\sigma}$ to be of value
in the intended applications mentioned in the introduction, namely questions
related to balance, detailed balance and joinings, where qualitative
differences between systems are relevant.

In Subsection \ref{OAfdSpronge} we also saw a case where $W_{\sigma}$ varies
continuously with the states (while keeping the dynamics the same), except for
a discontinuity when the states become equal. This is due to the balance
condition for the modular dynamics. As mentioned, this could be an indication
that $W$ is a more natural distance function than $W_{\sigma}$ for states
themselves (as opposed to systems), though at the cost of losing symmetry of
the distance function.

It would be of interest to investigate the ideas of this section in more
general terms, rather than just for our examples, to determine to what extent
the conclusion drawn here remain valid and to build further theory around them.

\end{document}